\documentclass[a4paper,11pt,twoside]{amsart}
\usepackage[english]{babel}
\usepackage[utf8]{inputenc}

\usepackage[a4paper,inner=3cm,outer=3cm,top=4cm,bottom=4cm,pdftex]{geometry}
\usepackage{fancyhdr}
\usepackage{color}
\usepackage{bold-extra}
\usepackage{ mathrsfs }

\usepackage{comment}
\usepackage{graphics}
\usepackage{aliascnt}
\usepackage[pdftex,citecolor=green,linkcolor=red]{hyperref}
\usepackage{mathtools}

\usepackage{amsmath}
\usepackage{amsfonts}
\usepackage{amssymb}
\usepackage{amsthm}
\usepackage{enumerate}

\newtheorem{theorem}{Theorem}[section]
\newtheorem{corollary}[theorem]{Corollary}

\newtheorem{lemma}[theorem]{Lemma}
\newtheorem{proposition}[theorem]{Proposition}
\theoremstyle{definition}
\newtheorem{definition}[theorem]{Definition}
\newtheorem{remark}[theorem]{Remark}

\numberwithin{equation}{section}

\renewcommand\P{\mathbf{P}}

\newcommand\E{\mathbb{E}}

\newcommand\Z{\mathbb{Z}}

\let\oldpmod\pmod
\renewcommand{\pmod}[1]{\hspace{-0.1cm}\oldpmod {#1}}

\renewcommand\d{\textnormal{ d}}
\renewcommand\:{\colon}

\begin{document}

\title{Quantitative asymptotics for polynomial patterns in the primes}

\author{Lilian Matthiesen}
\address{Mathematisches Institut, Georg-August-Universit\"{a}t G\"{o}ttingen, 37073 G\"{o}ttingen, Germany}
\email{lilian.matthiesen@mathematik.uni-goettingen.de}

\author{Joni Ter\"{a}v\"{a}inen}
\address{Department of Pure Mathematics and Mathematical Statistics, University of Cambridge, Cambridge CB3 0WB, UK}
\email{joni.p.teravainen@gmail.com}

\author{Mengdi Wang}
\address{\'{E}cole Polytechnique F\'{e}d\'{e}rale de Lausanne (EPFL), Lausanne, Switzerland}
\email{mengdi.wang@epfl.ch}

\begin{abstract} We prove quantitative estimates for averages of the von Mangoldt and M\"obius functions along polynomial progressions $n+P_1(m),\ldots, n+P_k(m)$ for a large class of polynomials $P_i$. The error terms obtained save an arbitrary power of logarithm, matching the classical Siegel--Walfisz error term. These results give the first quantitative bounds for the Tao--Ziegler polynomial patterns in the primes result. 

The proofs are based on a quantitative generalised von Neumann theorem of Peluse, a recent result of Leng on strong bounds for the Gowers uniformity of the primes, and analysis of a ``Siegel model'' for the von Mangoldt function along polynomial progressions.

\end{abstract}

\maketitle

\section{Introduction}

\subsection{M\"obius and von Mangoldt averages along polynomial progressions}

Given $k\in \mathbb{N}$ and polynomials $P_1,\ldots, P_k\in \mathbb{Z}[y]$, a \emph{polynomial progression} is a sequence of the form $n+P_1(m),\ldots, n+P_k(m)$ with $m,n\in \mathbb{Z}$.  In a seminal work, Tao and Ziegler~\cite{tao-ziegler1} proved that the primes contain infinitely many nonconstant polynomial progressions for any distinct polynomials $P_i$ satisfying $P_i(0)=0$ for all $1\leq i\leq k$. This generalised the Green--Tao theorem~\cite{gt} that corresponded to the polynomials $P_i$ being linear.

In a subsequent work, Tao and Ziegler~\cite{tao-ziegler2} went further by showing that, if $d=\max_{1\leq i\leq k}\deg(P_i)$ and $\Lambda$ denotes the von Mangoldt function\footnote{For convenience, we define $\Lambda$ on all of $\mathbb{Z}$ as an even function; thus $\Lambda(n)=\log p$ if $|n|=p^j$ for some prime $p$ and some $j\in \mathbb{N}$, and $\Lambda(n)=0$ otherwise.}, then there is an asymptotic formula for the count
\begin{align}\label{eq:count}
\frac{1}{N^{d+1}}\sum_{n\leq N^d}\sum_{m\leq N}\Lambda(n+P_1(m))\cdots \Lambda(n+P_k(m))     
\end{align}
of polynomial patterns in the primes, provided that $P_i-P_j$ has degree $d$ for all $1\leq i<j\leq k$. This generalised the work of Green, Tao and Ziegler~\cite{gt-linear},~\cite{gt-mobius},~\cite{gtz} that handled the case of linear polynomials $P_i$.

The results in~\cite{tao-ziegler1} are qualitative in nature, as they rely on the Gowers uniformity result for the von Mangoldt function from~\cite{gt-linear} (which is based on the qualitative Green--Tao--Ziegler~\cite{gtz} inverse theorem for the Gowers norms). Subsequently, a quantitative bound for the Gowers norms of the primes was established in~\cite{TT-JEMS} (using a quantitative inverse theorem for the Gowers norms due to Manners~\cite{manners}). Very recently, Leng~\cite{Leng-newII} made a substantial improvement to this result, obtaining arbitrary power of logarithm savings (using the quasipolynomial inverse theorem of Leng, Sah and Sawhney~\cite{LSS}). Leng applied his result in~\cite{Leng-newII} to give an asymptotic formula, with arbitrary power of logarithm savings, for the count~\eqref{eq:count} in the case of linear polynomials.

In the case where the $P_i$ are not all linear polynomials, we are only aware of quantitative results in the case $k\leq 2$ where classical Fourier analysis applies. Zhou~\cite{zhou} handled the case of the polynomials $\{0,y^d\}$ for any $d\in \mathbb{N}$ with savings of an arbitrary power of logarithm. See also~\cite{balestrieri},~\cite{BZ} for results of this type.

When it comes to a qualitative asymptotic formula for~\eqref{eq:count}, it is only known in the aforementioned case when $P_i-P_j$ has degree $d$ for all $1\leq i<j\leq k$ and a few additional cases with $k\leq 4$ (given in~\cite[Theorem 5 and Remark 2]{tao-ziegler2}).

The purpose of this paper is to give a quantitative asymptotic formula for the count~\eqref{eq:count} for a large class of polynomials. We obtain arbitrary power of logarithm savings, as in~\cite{Leng-newII}, which is the best one can hope for in the absence of an improvement to the Siegel--Walfisz theorem\footnote{For example in the case of the polynomials $0,y,2y$, it should be possible (e.g. following the approach in~\cite{TT-JEMS}) to extract the contribution of a Siegel zero to the main term, with the conclusion that any improvement to the error term beyond an arbitrary power of logarithm would lead to an improvement to Siegel's theorem and hence to the Siegel--Walfisz theorem. We will not pursue the details of this, however.}. The class of polynomials we can handle is the same for which qualitative results were obtained in~\cite{tao-ziegler2}.

\begin{theorem}[Quantitative polynomial patterns in the primes]\label{thm_mangoldt1}
Let $k,d\in \mathbb{N}$ be fixed,  and let  $P_1,\ldots, P_k\in \mathbb{Z}[y]$ be fixed. Let $\max_{1\leq i\leq k}\deg(P_i)=d$ and suppose that $\deg(P_i-P_j)=d$ for all $1\leq i<j\leq k$.
Then, for any $N\geq 3$ and $A>0$, we have  
\begin{align}\label{eq:count2}
\frac{1}{N^{d+1}}\sum_{n\leq N^{d}}\sum_{m\leq N}\Lambda(n+P_1(m))\cdots \Lambda(n+P_k(m))=\prod_{p}\beta_p+O_A((\log N)^{-A}),    
\end{align}
where
\begin{align*}
\beta_p\coloneqq \mathbb{E}_{m,n\in \mathbb{Z}/p\mathbb{Z}}\prod_{j=1}^k \Lambda_p(n+P_j(m)),   \end{align*}
and the local von Mangoldt function $\Lambda_p$ is given by
\begin{align}\label{eq:lambdap}
\Lambda_p(n)\coloneqq \frac{p}{p-1}1_{n\not \equiv 0\pmod p}.    
\end{align}
\end{theorem}

\textbf{Remarks.}

\begin{itemize}

    \item By~\cite[Lemma 9.5]{tao-ziegler1}, we have 
\begin{align}\label{eq:betap}
\beta_p=1+O(1/p^2),
\end{align}
since all large enough primes $p$ are ``good'' in the sense of~\cite[Definition 9.4]{tao-ziegler1}. This implies that the product $\prod_p \beta_p$ is convergent. Thus $\prod_p\beta_p>0$ unless there exists a prime $p$ such that for all $m,n\in \mathbb{Z}$ at least one of $n+P_1(m),\ldots, n+P_k(m)$ is divisible by $p$.

\item The result of Tao and Ziegler~\cite{tao-ziegler2} also works for multivariate polynomials. If Proposition~\ref{prop:GVNT3} below could be extended to multivariate polynomials, then it seems likely that Theorem~\ref{thm_mangoldt1} could also be extended to them.  
\end{itemize}

We also prove a similar quantitative bound for averages of the M\"obius function\footnote{For convenience, we define $\mu$ on all of $\mathbb{Z}$ as an even function; thus, $\mu(n)=(-1)^k$ if $|n|$ is the product of $k$ distinct primes and $\mu(n)=0$ if $|n|$ is not squarefree.} $\mu$ along a larger class of polynomial progressions.

\begin{theorem}[Quantitative polynomial patterns with M\"obius weight]\label{thm_mobius}
Let $k,d\in \mathbb{N}$ be fixed,  and let  $P_1,\ldots, P_k\in \mathbb{Z}[y]$ be fixed. Let $\max_{1\leq i\leq k}\deg(P_i)=d$. Suppose that there exists $1\leq \ell\leq k$ such that $\deg(P_{\ell}-P_i)=d$ for all $i\neq \ell$. 
Then, for any $N\geq 3$ and $A>0$, we have  
\begin{align*}
\left|\frac{1}{N^{d+1}}\sum_{n\leq N^d}\sum_{m\leq N}\mu(n+P_1(m))\cdots \mu(n+P_k(m))\right|\ll_A (\log N)^{-A}.    
\end{align*}
\end{theorem}

\textbf{Remarks.}
\begin{itemize}
    \item Theorem~\ref{thm_mobius} holds equally well with the Liouville function $\lambda$ in place of the M\"obius function $\mu$. This follows from the same proof along with Remark~\ref{rmk:liouville}.     

    \item Note that in Theorem~\ref{thm_mobius} we can allow some of the polynomials to differ only by a constant (or even to be equal). Thus, for example, if $P_1,\ldots, P_{k-1}\in \mathbb{Z}[y]$ are any polynomials of degree at most $d-1$, then $\{P_1,\ldots, P_{k-1},y^d\}$ is a valid collection for this theorem.  
\end{itemize}

As far as we are aware, Theorem~\ref{thm_mobius} gives the first quantitative bounds for the averages of the M\"obius function along polynomial progressions with non-linear polynomials (and more than two polynomials). Tao and Ziegler stated that their method in~\cite{tao-ziegler2} works for giving qualitative cancellation for averages of $\mu$ along polynomial progressions in the same cases where they handled the von Mangoldt weight. Qualitative cancellation for averages of $\mu$ over general polynomial progressions (with no two polynomials differing by a constant) was subsequently proved in~\cite[Theorem 1.10]{MRTTZ}.  

\subsection{One-dimensional averages}

It turns out that we can give quantitative estimates not only for the double averages~\eqref{eq:count}, but also for the single averages
\begin{align*}
 \sum_{n\leq N^{d}}\Lambda(n+P_1(m))\cdots \Lambda(n+P_k(m))\quad \textnormal{and}\quad     \sum_{m\leq N}\Lambda(n+P_1(m))\cdots \Lambda(n+P_k(m))
\end{align*}
for almost all $m\in [N]$ (respectively, for almost all $n\in [N^d]$). This follows from the following theorem, from which Theorems~\ref{thm_mangoldt1} and~\ref{thm_mobius} will be deduced. 

\begin{theorem}\label{thm_general}
Let $k,d\in \mathbb{N}$ be fixed,  and let  $P_1,\ldots, P_k\in \mathbb{Z}[y]$ be fixed.  Let $\max_{1\leq i\leq k}\deg(P_i)=d$, $N\geq 3$ and $A>0$. 
\begin{enumerate}
    \item Suppose that there exists $1\leq \ell\leq k$ such that $\deg(P_{\ell}-P_i)=d$ for all $i\neq \ell$. Then we have  
\begin{align*}
\frac{1}{N^{d+1}}\sum_{m\leq N}\left|\sum_{n\leq N^{d}}\mu(n+P_1(m))\cdots \mu(n+P_k(m))\right|\ll_A (\log N)^{-A},    
\end{align*}
and if such an index $l$ can be chosen with $\deg P_l=d$, then
\begin{align*}
\frac{1}{N^{d+1}}\sum_{n\leq N^d}\left|\sum_{m\leq N}\mu(n+P_1(m))\cdots \mu(n+P_k(m))\right|\ll_A (\log N)^{-A}.        
\end{align*}
\item Suppose that   $\deg(P_i-P_j)=d$ for all $1\leq i<j\leq k$. Then there exists a number $C_0=C_0(P_1,\dots,P_k)$ sufficiently large so that $P_i(m) \neq P_j(m)$ whenever $m\geq C_0$ such that
\begin{align*}
\frac{1}{N}\sum_{C_0 \leq m\leq N}\left|N^{-d}\sum_{n\leq N^{d}}\Lambda(n+P_1(m))\cdots \Lambda(n+P_k(m))-\prod_{p}\beta_p(m)\right|\ll_A (\log N)^{-A};    
\end{align*}
and if additionally all the $P_i$ are nonconstant and of degree $d$, there exists a constant $B=B(P_1,\dots,P_k,A)$ large enough, such that
\begin{align*}
\frac{1}{N^{d}}\sum_{n\leq N^d}\left|\frac1N\sum_{m\leq N}\Lambda(n+P_1(m))\cdots \Lambda(n+P_k(m))-\prod_{p\leq (\log N)^B}\beta_p'(n)\right|\ll_A (\log N)^{-A},    
\end{align*}
where
\begin{align}\label{eq:betapm}
\beta_p(m)\coloneqq \mathbb{E}_{n\in \mathbb{Z}/p\mathbb{Z}}\prod_{j=1}^k \Lambda_p(n+P_j(m)),\quad \beta_p'(n)\coloneqq \mathbb{E}_{m\in \mathbb{Z}/p\mathbb{Z}}\prod_{j=1}^k \Lambda_p(n+P_j(m)).   
\end{align}
\end{enumerate}
\end{theorem}

Theorem~\ref{thm_general}(1) implies Theorem~\ref{thm_mobius} directly by the triangle inequality. The implication from Theorem~\ref{thm_general}(2) to Theorem~\ref{thm_mangoldt1} is shown in Section~\ref{sec:lambda}. Also in Section~\ref{sec:lambda}, we prove Theorem~\ref{thm_general}(2) after giving an overview of the proof and the role of the Siegel model in the proof.  

As an immediate corollary of Theorem~\ref{thm_general} and Markov's inequality, we have the following almost-all result.

\begin{corollary}[An averaged multidimensional Bateman--Horn result]\label{cor:BH}
Let $k,d\in \mathbb{N}$ be fixed,  and let  $P_1,\ldots, P_k\in \mathbb{Z}[m]$ be fixed and none of them is a constant polynomial. Let $\deg(P_i)=d$ for all $1\leq i\leq k$. Suppose that $\deg(P_i-P_j)=d$ for all $1\leq i<j\leq k$. Let $N\geq 3$ and $A>0$. Then there exists $B=B(P_1,\dots,P_k,A)$ sufficiently large and such that
\begin{align*}
\left|\frac 1N\sum_{m\leq N}\Lambda(P_1(m)+n)\cdots \Lambda(P_k(m)+n)-\prod_{p\leq (\log N)^B}\beta_p'(n)\right|\ll_A (\log N)^{-A},    
\end{align*} 
for all but $\ll_A N^d/(\log N)^{A}$ integers $n\in [N^d]$. 
\end{corollary}

Corollary \ref{cor:BH} may be viewed as an averaged analogue of the Bateman--Horn conjecture, with a truncated singular series and a quantitative error term. Exceptional shifts are necessary, since some of the polynomials $P_i(m)+n$ may become reducible. Specialising to $k=1$, Corollary~\ref{cor:BH} extends an earlier result of Zhou~\cite{zhou}, which handled the case where $P_1(y)=y^{\ell}$ for some $\ell \in \mathbb{N}$. 

\textbf{Note added in November 2026:} Following observations communicated to us by an anonymous reader, we have corrected the second assertion of  Theorem~\ref{thm_general}(2) and Corollary~\ref{cor:BH} as stated in the published version. Both statements now include the hypothesis $\deg P_i=d$ for all $1\leq i\leq k$. These corrections concern cases where one of the polynomials has degree less than $d$, including the possibility of a constant polynomial.  These changes only concern the edge cases  and do not influence Theorem \ref{thm_mangoldt1} or its arithmetic input in Proposition \ref{prop:leng2}.

\subsection{Acknowledgements} 
LM was supported by the Swedish Research Council (VR) grant
2021-04526.
JT was supported by funding from European Union's Horizon
Europe research and innovation programme under Marie Sk\l{}odowska-Curie grant agreement no.
101058904. MW was supported by the Academy of Finland grant no. 333707. Part of this material is based upon work supported by the Swedish Research Council under grant no. 2021-06594 while the authors were in residence at Institut Mittag-Leffler in Djursholm, Sweden, during Spring 2024.

\subsection*{Declaration of AI use}
The previously published versions of this paper were written entirely  independently of any AI tools. After an anonymous reader alerted us by email to issues involving certain edge cases, we used ChatGPT Pro to help check for errors and suggest corrections. The authors take full responsibility for the content of this paper.

\section{Notation and preliminaries}\label{sec:notation}

\subsection{Asymptotic and averaging notation}

We use the usual Landau and Vinogradov asymptotic notations $O(\cdot), o(\cdot), \ll, \gg$. Thus, we write $X\ll Y$, $X=O(Y)$ or $Y\gg X$ if we have $|X|\leq CY$ for some constant $C$. We write $X\asymp Y$ if $X\ll Y\ll X$. We write $X=o(Y)$ as $N\to \infty$ if $|X|\leq c(N)Y$ for some function $c(N)\to 0$ as $N\to \infty$. If we attach subscripts to these notations, then the implied constants may depend on these subscripts. Thus, for example,  $X\ll_A Y$ signifies that $|X|\leq C_AY$ for some $C_A>0$ depending on $A$. 

For a set $S$, we write $1_{S}(x)$ for its indicator function, i.e. a function that equals to $1$ for $x\in S$ and equals to $0$ otherwise. Similarly, if $P$ is a proposition, the quantity $1_{P}$ equals to $1$ if $P$ is true and $0$ if $P$ is false. 

For a nonempty finite set $A$ and a function $f\:A\to \mathbb{C}$, we use the averaging notation
\begin{align*}
\mathbb{E}_{a\in A}f(a)\coloneqq \frac{1}{|A|}\sum_{a\in A}f(a).    
\end{align*} 

For a real number $N>0$, we denote $[N]\coloneqq [1,N]\cap \mathbb{N}$. For $q\in \mathbb{Z}$, we also denote $q[N]\coloneqq \{qn\colon \,\, n\in [1,N]\cap \mathbb{N}\}$. 

We write $\|x\|$ for the distance from the real number $x$ to the nearest integer. We use the standard notation $e(x)\coloneqq e^{2\pi i x}$.

Given integers $m,n$, we write $(m,n)$ for their greatest common divisor and $[m,n]$ for their least common multiple. We use $\tau(n)$ to denote the divisor function and $\Omega(n), \omega(n)$ to denote the number of prime divisors of $n$, with and without multiplicities, respectively.

\subsection{Gowers norms}

For $f\: \mathbb{Z}\to \mathbb{C}$ with finite support and $s\in \mathbb{N}$, we define the unnormalised $U^s$ Gowers norm of $f$ as
\begin{align*}
\|f\|_{\widetilde{U}^s(\mathbb{Z})}\coloneqq\left(\sum_{x,h_1,\ldots, h_s\in \mathbb{Z}}\prod_{\omega\in \{0,1\}^s}\mathcal{C}^{|\omega|}f(x+\omega\cdot (h_1,\ldots, h_s))\right)^{1/2^s},   
\end{align*}
where $\mathcal{C}(z)=\overline{z}$ is the complex conjugation operator and for a vector $(\omega_1, \ldots, \omega_s)$ we write $|\omega|\coloneqq |\omega_1| +\cdots+ |\omega_s|$. 
For an interval $I\subset \mathbb{R}$ with $|I|\geq 1$, we further define the $U^s(I)$ Gowers norm of a function $f\:\mathbb{Z}\to \mathbb{C}$ as
\begin{align*}
\|f\|_{U^s(I)}\coloneqq \frac{\|f1_{I}\|_{\widetilde{U}^s(\mathbb{Z})}}{\|1_{I}\|_{\widetilde{U}^s(\mathbb{Z})}}.    
\end{align*}
It is well known (see for example~\cite[Appendix B]{gt-linear}) that for $s\geq 2$ the $U^s(I)$ norm is indeed a norm and for $s=1$ it is a seminorm. When $I=[N_1,N_2]$, we abbreviate $\|f\|_{U^s(I)}=\|f\|_{U^s[N_1,N_2]}$ and if $N_1=1$ we further abbreviate this as $\|f\|_{U^s[N_2]}$.  

The $\widetilde{U}^s(\mathbb{Z})$ norms obey the Gowers--Cauchy--Schwarz inequality (see for example~\cite[(4.2)]{TT-JEMS}), which states that, for any functions $(f_{\omega})_{\omega\in \{0,1\}^s}$ from $\mathbb{Z}$ to $\mathbb{C}$ with finite support, we have
\begin{align}\label{eq:gcs}
\left|\sum_{x,h_1,\ldots, h_s\in \mathbb{Z}}\prod_{\omega\in \{0,1\}^s}f_{\omega}(x+\omega\cdot (h_1,\ldots, h_s))\right|\leq \prod_{\omega\in \{0,1\}^s}\|f_{\omega}\|_{\widetilde{U}^s(\mathbb{Z})}.  
\end{align}

For  $H\geq 1$, we define the weight
\begin{align}\label{eq:muH}
\mu_H(h) \coloneqq \frac{ |\{(h_1,h_2)\in [H]^2\: h_1-h_2=h\}|}{\lfloor H\rfloor ^2},    
\end{align}
where $\lfloor x\rfloor$ denotes the floor of $x$. This weight will appear naturally when applying the Cauchy--Schwarz inequality to polynomial averages.

\subsection{Vinogradov's Fourier expansion}

We will make use of a Fourier approximation for the indicator function of an interval that is due to Vinogradov.

\begin{lemma}\label{le:vinogradov} For any $-1/2< \alpha<\beta< 1/2$ and $\eta\in  (0,\min\{1/2-\|\alpha\|,1/2-\|\beta\|,\|\alpha-\beta\|/2\})$, there exists a $1$-periodic function $g\:\mathbb{R}\to [0,1]$ such that the following hold.
\begin{enumerate}
    \item $g(x)=1$ for $x\in [\alpha+\eta, \beta-\eta]$ and $g(x)=0$ for $x\in [-1/2,1/2]\setminus [\alpha-\eta,\beta+\eta]$.

    \item For some $|c_j|\ll \min\{1,\frac{1}{|j|}, \frac{1}{\eta j^2}\}$, we have for all $x\in \mathbb{R}$ the Fourier expansion
    \begin{align*}
     g(x)=\beta-\alpha-\eta+\sum_{|j|>0}c_je(jx).   
    \end{align*}
    \item For any $K\geq 1$, we have
    \begin{align*}
     \sum_{|j|>K}|c_j|\leq \frac{10\eta^{-1}}{K}.   
    \end{align*}
\end{enumerate}
\end{lemma}

\begin{proof}
See~\cite[Chapter 1, Lemma 12]{vinogradov} (with $r=1$, $\Delta=2\eta$, $(\alpha,\beta)\to (\alpha+\eta,\beta-\eta)$ there).   
\end{proof}

\section{Gowers norm control of polynomial averages}

For the proofs of the main theorems, 
we need to know that weighted polynomial averages (with maximal degree $d$) are controlled in terms of the $U^s[CN^d]$ norm of one of the weights for some constant $C$, assuming that the corresponding polynomials $P_i$ satisfy the assumption of Theorem~\ref{thm_general}. This follows from a result that slightly extends the work of Peluse~\cite{Peluse-FMP}. 

\begin{proposition}\label{prop:GVNT3}
Let $k,d\in \mathbb{N}$, $C\geq 2$ and let $P_1,\ldots, P_k\in \mathbb{Z}[y]$ be nonconstant polynomials satisfying $\deg P_1\leq \ldots \leq \deg P_k=d$. Suppose that $\deg(P_k-P_i)=d$ for all $1\leq i\leq k-1$.
Let $c_i$ be the leading coefficient of $P_i$, and suppose that all of the coefficients of $P_1,\ldots, P_k$ are at most $C|c_k|$ in modulus. Suppose also that $|c_i|\geq |c_k|/C$ whenever $\deg(P_i)=d$ and that $|c_k|\leq N^{\eta}$ with $\eta>0$ small enough in terms of $d$.

Let $N\geq 1$, and let $f_0,\ldots, f_k\: \mathbb{Z}\to \mathbb{C}$ be functions supported on $[-CN^d,CN^d]$ with $|f_i|\leq 1$ for all $0\leq i\leq k$. Also let $\theta\colon [N]\to \mathbb{C}$ satisfy $|\theta|\leq 1$. Then, for some natural number $s\ll_{d,k} 1$ and some $1\leq K\ll_{d,k}1$, we have
\begin{align}\label{eq:GVNT3}\begin{split}
\left|\frac{1}{N^{d+1}}\sum_{n\in \mathbb{Z},m\in [N]}\theta(m)f_0(n)\prod_{j=1}^k f_j(n+P_j(m))\right|\ll_{C,d,k} |c_k|^{K}\left(N^{-1}+\|f_k\|_{U^s[-CN^d,CN^d]}\right)^{1/K}.   
\end{split}
\end{align}
\end{proposition}

\begin{remark}
Compared with~\cite[Theorem 6.1]{Peluse-FMP}, we have an additional weight $\theta$ on the $m$ variable and we bound the averages in terms of the Gowers uniformity norms rather than Gowers box norms.     
\end{remark}

\begin{proof} \textbf{Reduction to the case $\theta\equiv 1$.} We first reduce matters to proving the case $\theta\equiv 1$. Suppose that this case has been proved, and consider the general case. Let $S$ denote the left-hand side of~\eqref{eq:GVNT3}. Then by the Cauchy--Schwarz inequality we have
\begin{align*}
|S|^2\leq \frac{1}{N^{2d+1}}\sum_{m\in [N]}\left|\sum_{n\in \mathbb{Z}}f_0(n)\prod_{j=1}^k f_j(n+P_j(m))\right|^2.   
\end{align*}
Expanding out the square, the right-hand side becomes
\begin{align*}
\frac{1}{N^{2d+1}}\sum_{|h|\leq 2CN^d}\sum_{m\in [N]}\sum_{n\in \mathbb{Z}}\Delta_hf_0(n)\prod_{j=1}^k \Delta_hf_j(n+P_j(m)),
\end{align*}
where $\Delta_hf(x)\coloneqq \overline{f(x+h)}f(x)$. Applying the case $\theta\equiv 1$ to the inner sums here, we see that for some $1\leq s,K\ll_{d,k}1$ we have
\begin{align*}
|S|^2\ll_{C,d,k}  \frac{|c_k|^{K}}{N^{d}}\sum_{|h|\leq 2CN^d}\left(N^{-1}+\|\Delta_hf_k\|_{U^s[-CN^d,CN^d]}\right)^{1/K}.   
\end{align*}
Using $(a+b)^{1/K}\leq a^{1/K}+b^{1/K}$ for $a,b\geq 0$ and H\"older's inequality, we see that
\begin{align*}
|S|^2\ll_{C,d,k}  |c_k|^{K}\left(N^{-1/K}+\left(\frac{1}{N^d}\sum_{|h|\leq 2CN^d}\|\Delta_hf_k\|_{U^s[-CN^d,CN^d]}^{2^s}\right)^{1/(2^sK)}\right).
\end{align*}
Expanding out the definition of the $U^s[-CN^d,CN^d]$ norm, we conclude that
\begin{align*}
|S|^2\ll_{C,d,k}  |c_k|^{K}\left(N^{-1/K}+\|f_k\|_{U^{s+1}[-CN^d,CN^d]}^{2/K}\right).    
\end{align*}
Now the claim~\eqref{eq:GVNT3} follows (with $s$ replaced by $s+1$ and $K$ by $2K$) from the power mean inequality $(a^{1/K}+b^{1/K})/2\leq ((a+b)/2)^{1/K}$ for $a,b\geq 0$. 

\textbf{Proof of the case $\theta\equiv 1$.} We may also assume that $C$ is large enough in terms of $d,k$. 

We are going to apply~\cite[Theorem 6.1]{Peluse-FMP}. Let $s\ll_{d,k}1$ be as in that theorem (we may assume that $s\geq 2$). Also let $B=B(d,k)$ be large enough.
Let 
\begin{align}\label{eq:deltadef}
\delta\coloneqq \frac{1}{C^2|c_k|N^{d+1}}\sum_{n\in \mathbb{Z},m\in [N]}f_0(n)\prod_{j=1}^k f_j(n+P_j(m)).    
\end{align}

We may suppose that
\begin{align*}
N\geq B\left(\frac{|c_k|}{\delta^B}\right)^{B},    
\end{align*}
since otherwise solving this inequality for $\delta$ the claim of the proposition follows immediately. Since $|f_j|\leq 1$, we also have $\delta\leq 1/C^2$.

Now we apply~\cite[Theorem 6.1]{Peluse-FMP}\footnote{That theorem is stated for functions supported on intervals of the form $[X]$, but the same proof works for functions supported on an interval of the form $[-X,X]$.} with $(M,N,C)\to (N,C^2|c_k|N^d,C^{3})$ there (noting that the assumptions of that theorem are satisfied; in particular $|c_i|N^d/(C^2|c_k|N^{d})\in [C^{-3},C^3]$ for all $i$). Recalling the definition of the Gowers box norms from~\cite[Section 2]{Peluse-FMP}, we conclude that for some $t=t(d,k)$ (which may be assumed to be much smaller than $B$) and for $\delta'= \delta^{t}$  we have
\begin{align}\label{eq:square}
\|f_k\|_{\square^{s}_{Q_1,\ldots, Q_s}([C^2|c_k|N^d])}\gg_{C,d,k} \delta',
\end{align}
where each $Q_i$ equals to $r_i[\delta'N^d]\subset r_i\mathbb{Z}$ for some integer\footnote{The numbers $r_i$ are of the form $c_k$ or $c_k-c_i$ for some $i$ with $\deg P_i=d$; crucially, these numbers cannot be zero by our assumption on $P_i$.} $0<|r_i|\leq 2d!C|c_k|$. To simplify notation, we assume that all $r_i$ are positive. Unwrapping the definition of the box norm,~\eqref{eq:square} yields
\begin{align*}
&\Bigg|\frac{1}{C^2|c_k|N^{d(2s+1)}}\sum_{n,h_1^{(0)},h_1^{(1)},\ldots, h_s^{(0)},h_s^{(1)}\in \mathbb{Z}}\prod_{j=1}^{s}1_{h_j^{(0)},h_j^{(1)}\in [r_j\delta'N^d]\cap r_j\mathbb{Z}}\\
&\quad \quad\cdot \prod_{\omega\in \{0,1\}^s}\mathcal{C}^{|\omega|}f_k(n+h_1^{(\omega_1)}+\cdots +h_k^{(\omega_k)})\Bigg|\\
&\quad \gg_{C,d,k} (\delta')^{2^s}. 
\end{align*}
Making the changes of variables $n'=n-(h_1^{(0)}+\cdots +h_k^{(0)})$, $h_j=h_j^{(1)}-h_j^{(0)}$ and using the notation~\eqref{eq:muH} and the trivial bound $1\geq 1/(C^2|c_k|)$, we see that
\begin{align}\label{eq:deltaB}\begin{split}
&\Bigg|\frac{1}{N^{d(s+1)}}\sum_{\substack{n'\in [-CN^d,CN^d]\\h_1,\ldots, h_s\in [-CN^d,CN^d]}}\prod_{j=1}^{s}1_{h_j\equiv 0\pmod{r_j}}\delta'N^d\mu_{\delta'N^d}\left(\frac{h_j}{r_j}\right)\\
&\quad \quad \cdot\prod_{\omega\in \{0,1\}^s}\mathcal{C}^{|\omega|}f_k(n'+\omega\cdot (h_1,\ldots, h_s))\Bigg|\\
&\quad \gg_{C,d,k} (\delta')^{2^s}. 
\end{split}
\end{align}

By the orthogonality of characters, we have 
\begin{align}\label{eq:ortho}
1_{h_j\equiv 0\pmod{r_j}}=\frac{1}{r_j}\sum_{1\leq a\leq r_j}e\left(\frac{ah_j}{r_j}\right).
\end{align}
Moreover, for any $H\in\mathbb{N}$ and $x\in [-H,H]\cap\mathbb{Z}$ we have the Fourier expansion
\begin{align}\label{eq:muexpansion}
H\mu_{H}(x)&=1-\left|\frac{x}{H}\right|=\frac{1}{2}+\frac{2}{\pi^2}\sum_{m\equiv 1\pmod 2}\frac{1}{m^2}e\left(\frac{x}{2H}m\right).
\end{align}
Inserting~\eqref{eq:ortho},~\eqref{eq:muexpansion} into~\eqref{eq:deltaB} and applying the triangle inequality, for some $\alpha_1,\ldots, \alpha_s\in \mathbb{R}$ we obtain
\begin{align}\label{eq:rj}\begin{split}
&\left|\frac{1}{N^{d(s+1)}}\sum_{n,h_1,\ldots, h_s\in [-2CN^d,2CN^d]}\prod_{j=1}^{s}1_{[-r_j\delta'N^d, r_j\delta'N^d]}(h_j)e(\alpha_j h_j)\right.\\
&\quad\quad  \left.\cdot \prod_{\omega\in \{0,1\}^s}\mathcal{C}^{|\omega|}f_k1_{[-CN^d,CN^d]}(n+\omega\cdot \mathbf{h})\right|\\
&\quad \gg_{C,d,k} (\delta')^{2^s}. 
\end{split}
\end{align}

Observe that since $\delta\leq 1/C^2$ and $C$ is large enough in terms of $d,k$, we have $r_j\delta'\leq 1/2$ for all $1\leq j\leq s$. Note then that for any $c\in (0,1)$ and $h\in [-2CN^d,2CN^d]$ we have 
$$1_{[-cN^d,cN^d]}(h)=1_{\|h/(4CN^d)\|\leq c/(4C)}.$$
Hence, by Vinogradov's Fourier expansion (Lemma~\ref{le:vinogradov}), for any $1\leq j\leq s$ there exists a trigonometric polynomial $T_j(x)=\sum_{0\leq |m|\leq J}\gamma_{m,j}e(mx/(4CN^d))$ with $J=O_C((\delta')^{-10\cdot 2^s})$ such that we have $|\gamma_{m,j}|\ll 1$ for all $m$ and 
\begin{align}\label{eq:Tj}
N^{-d}\sum_{h\in [-2CN^d,2CN^d]}|1_{[-r_j\delta'N^d, r_j\delta'N^d]}(h)-T_j(h)|\ll (\delta')^{5\cdot 2^s}+N^{-d}.    
\end{align}
Replacing $1_{[-r_j\delta'N^d,r_j\delta'N^d]}(h_j)$ with $T_j(h_j)$ in~\eqref{eq:rj}, and crudely using the triangle inequality, the $1$-boundedness of $f_k$ and~\eqref{eq:Tj} to estimate the error from this replacement, we see that   
\begin{align*}
&\left|\frac{1}{N^{d(s+1)}}\sum_{n,h_1,\ldots, h_s\in [-2CN^d,2CN^d]}\prod_{j=1}^{s}T_j(h_j)e(\alpha_jh_j)\cdot \prod_{\omega\in \{0,1\}^s}\mathcal{C}^{|\omega|}f_k1_{[-CN^d,CN^d]}(n+\omega\cdot \mathbf{h})\right|\\
&\quad\gg_{C,d,k} (\delta')^{2^s}. 
\end{align*}

Writing out $T_j$ as a trigonometric polynomial and applying the pigeonhole principle, we see that for some real numbers $\alpha_1',\ldots, \alpha_s'$ we have 
\begin{align*}
&\left|\frac{1}{N^{d(s+1)}}\sum_{n,h_1,\ldots, h_s\in [-2CN^d,2CN^d]}\prod_{j=1}^{s}e(\alpha_j' h_j)\cdot \prod_{\omega\in \{0,1\}^s}\mathcal{C}^{|\omega|}f_k1_{[-CN^d,CN^d]}(n+\omega\cdot \mathbf{h})\right|\\
&\quad \gg_{C,d,k} (\delta')^{100s2^s}. 
\end{align*}
Using $e(\alpha_j'h_j)=e(\alpha_j'(n+h_j))e(-\alpha_j' n)$, this can be written as 
\begin{align*}
&\left|\frac{1}{N^{d(s+1)}}\sum_{n,h_1,\ldots, h_s\in [-2CN^d,2CN^d]}\prod_{\omega\in \{0,1\}^s}\mathcal{C}^{|\omega|}f_{k,\omega}1_{[-CN^d,CN^d]}(n+\omega\cdot \mathbf{h})\right|\gg_{C,d,k} (\delta')^{100s2^s}, 
\end{align*}
where each $f_{k,\omega}$ is of the form $f_{k,\omega}(n)=f_k(n)e(\alpha_{\omega}n)$ for some real number $\alpha_{\omega}$. 

Now, by the Gowers--Cauchy--Schwarz inequality~\eqref{eq:gcs} and the simple fact that 
$$\|f_ke(\alpha\cdot)\|_{U^s([-CN^d,CN^d])}=\|f_k\|_{U^s([-CN^d,CN^d])}$$
for any real number $\alpha$ and any $s\geq 2$, we conclude that 
\begin{align*}
\|f_k\|_{U^s([-CN^d,CN^d])}^{2^s}\gg_{C,d,k} (\delta')^{100s2^s}.  
\end{align*}
Recalling that $\delta'=\delta^{t}$ and taking $K= 100st$, this yields
\begin{align*}
\delta\ll_{C,d,k} \|f_k\|_{U^s[-CN^d,CN^d]}^{1/K}.
\end{align*}
Taking into account the definition of $\delta$ in~\eqref{eq:deltadef}, the proof is now complete.
\end{proof}

\section{Gowers norms of the M\"obius and von Mangoldt functions}

A key number-theoretic input needed for the proofs of our main theorems is a bound for the Gowers norms of the M\"obius and von Mangoldt functions saving an arbitrary power of logarithm, due to Leng~\cite{Leng-newII}.

\begin{proposition}[Quantitative $U^k$-uniformity of $\mu$]\label{prop:leng} Let $k\in \mathbb{N}$, $A>0$ and $N\geq 3$. Then we have
\begin{align*}
\|\mu\|_{U^k[N]}\ll_A (\log N)^{-A}.     
\end{align*}
\end{proposition}

\begin{remark}\label{rmk:liouville} Proposition~\ref{prop:leng} holds also with the Liouville function $\lambda$ in place of the M\"obius function $\mu$; this is seen from the identity 
\begin{align*}
\lambda(n)=\sum_{d^2\mid n}\mu(n/d^2).    
\end{align*}    
\end{remark}

\begin{proof}[Proof of Proposition~\ref{prop:leng}] This follows by combining~\cite[Theorem 6]{Leng-newII} with the triangle inequality for the Gowers norms and~\cite[Theorem 2.5]{TT-JEMS} (where we can use Siegel's bound to lower bound the size of a Siegel modulus instead of a weaker effective bound).
\end{proof}

For the von Mangoldt function, we end up needing a stronger Gowers uniformity result, saving more than an arbitrary power of logarithm (as in~\cite{Leng-newII}). This necessitates taking into account the contribution of a possible Siegel zero. We introduce some definitions for this.

\begin{definition}
Let $c_0=1/1000$ and $Q\geq 2$. A real number $\widetilde{\beta}\in (0,1)$ is called \emph{a Siegel zero of level $Q$} if there exists a primitive Dirichlet character $\widetilde{\chi}$ of conductor $\leq Q$ such that $L(\widetilde{\beta},\widetilde{\chi})=0$ and $\widetilde{\beta}\geq 1-c_0/\log Q$. The character $\widetilde{\chi}$ is then called a \emph{Siegel character of level $Q$}. 
\end{definition}

By the Landau--Page theorem (see e.g.~\cite{pintz}), we know that if a Siegel zero of level $Q$ exists, it is unique and simple, and the corresponding Siegel character $\widetilde{\chi}$ is a unique non-principal real character. 

\begin{definition}[The Siegel model]\label{def:Siegelmodel}
Let $w\geq 2$ and $W=\prod_{p\leq w}p$. Let $\widetilde{\beta}$ be the Siegel zero of level $w$ and $\widetilde{\chi}$ be the Siegel character of level $w$, if they exist; otherwise let $\widetilde{\beta}=1$ and $\widetilde{\chi}=0$. Define for $n\in \mathbb{Z}$ the \emph{Siegel model}
\begin{align*}
\widetilde{\Lambda}_{W}(n)\coloneqq \Lambda_W(n)\left(1-\widetilde{\chi}(|n|)|n|^{\widetilde{\beta}-1}\right),    
\end{align*}
where
\begin{align}\label{eq:lambdaw}
\Lambda_W(n)\coloneqq \prod_{p\mid W}\Lambda_p(n),    
\end{align}
with $\Lambda_p$ as in~\eqref{eq:lambdap}.
\end{definition}

With this notation, we have the following version of the prime number theorem in arithmetic progressions with a Siegel correction. 

\begin{proposition}\label{prop:landaupage}
Let $N\geq 2$. Let $2\leq w\leq \exp((\log N)^{1/2})$, and let $\widetilde{\Lambda}_{W}$ be the Siegel model as in Definition~\ref{def:Siegelmodel}. Then
\begin{align}\label{eq:lambda2}
\max_{q\leq w}\,\max_{a\pmod q} \Bigg|\sum_{\substack{m\leq N\\m\equiv a\pmod q}}(\Lambda(m)-\widetilde{\Lambda}_{W}(m))\Bigg|\leq N\exp(-c(\log N)^{1/2})  
\end{align}
for some absolute constant $c>0$. 
\end{proposition}

\begin{proof}
Let $1\leq a\leq q\leq w$ be chosen so that the maximum in~\eqref{eq:lambda2} is attained. Denote the sum on the left-hand side of by $S(N;q,a)$. If a prime $p$ divides $(a,q)$, then $p\leq w$, so $\widetilde{\Lambda}_{W}(m)=0$ for all $m\equiv 0\pmod{p}$ and $\Lambda(m)=0$ for all $m\equiv 0\pmod{p}$ that are not of the form $p^j$ with $j\in \mathbb{N}$. Hence we have $|S(N;q,a)|\ll (\log N)^2$, which is much stronger than the claimed bound. This means that we can assume from now on that $(a,q)=1$.

Expanding the indicator $1_{m\equiv a\pmod q}$ in terms of Dirichlet characters, we find
\begin{align}\label{eq:chibeta}\begin{split}
&\sum_{\substack{m\leq N\\m\equiv a\pmod q}}\widetilde{\Lambda}_{W}(m)=\frac{1}{\varphi(q)}\sum_{\chi\pmod q}\overline{\chi}(a)\sum_{m\leq N}\widetilde{\Lambda}_{W}(m)\chi(m).
\end{split}
\end{align}

We claim that for any character $\chi\pmod q$ we have
\begin{align}\label{eq:psi}
\sum_{m\leq N}\Lambda_W(m)\chi(m)=\begin{cases}N+O(N\exp(-(\log N)^{1/2}/100)),\quad &\textnormal{if}\quad \chi\quad \textnormal{is principal}.\\
O(N\exp(-(\log N)^{1/2}/100)),\quad &\textnormal{if}\quad \chi\quad \textnormal{is non-principal}.
\end{cases}
\end{align}
Indeed, if $\chi\pmod q$ is principal, the fundamental lemma of sieve theory (\cite[Theorem 6.9]{opera}) gives
\begin{align*}
\sum_{m\leq N}\Lambda_W(m)\chi(m)=\sum_{m\leq N}\Lambda_W(m)=N+O(N(\log w)e^{-s}),    
\end{align*}
where $s$ is such that $w^s=N^{1/10}$, so in particular $s\geq (\log N)^{1/2}/10$. This gives the first part of the claim~\eqref{eq:psi}. If $\chi\pmod q$ is non-principal, in turn, the fundamental lemma of sieve theory again gives
\begin{align*}
\sum_{m\leq  N}\Lambda_W(m)\chi(m)&=\sum_{\substack{b\pmod q\\(b,q)=1}}\chi(b)\,\sum_{\substack{m\leq N\\m\equiv b\pmod q}}\Lambda_W(m)\\
&=\frac{N}{\varphi(q)}\, \,\sum_{\substack{b\pmod q\\(b,q)=1}}\chi(b)+O(N(\log w)e^{-s})\\
&=O(N(\log w)e^{-s}),    
\end{align*}
where $s$ is as above, so in particular $s\geq (\log N)^{1/2}/10$. This completes the proof of~\eqref{eq:psi}.

Using~\eqref{eq:psi}, we obtain
\begin{align*}
\sum_{m\leq N}\widetilde{\Lambda}_{W}(m)\chi(m)=N1_{\chi \textnormal{ principal}}-\frac{N^{\widetilde{\beta}}}{\widetilde{\beta}}1_{\chi\widetilde{\chi}\textnormal{ principal}}+O(N\exp(-(\log N)^{1/2}/100)).  
\end{align*}
Inserting this into~\eqref{eq:chibeta}, we see that
\begin{align*}
\sum_{\substack{m\leq N\\m\equiv a\pmod q}}\widetilde{\Lambda}_{W}(m)=\frac{N}{\varphi(q)}\left(1-1_{\widetilde{q}\mid q}\widetilde{\chi}(a)\frac{N^{\widetilde{\beta}-1}}{\widetilde{\beta}}\right)+O(N\exp(-(\log N)^{1/2}/100)).    
\end{align*}
Now the claim~\eqref{eq:lambda2} follows from~\cite[Theorem 5.27]{iw-kow} (which is a strengthening of the prime number theorem in arithmetic progressions taking the possible Siegel zero into account). 
\end{proof}

We are now ready to prove the key Gowers uniformity property of the von Mangoldt function used in this paper.

\begin{proposition}[Quantitative $U^k$-uniformity of $\Lambda$]\label{prop:leng2}
 For each $k\in \mathbb{N}$, there exists $c_k>0$ such that the following holds. Let $N\geq 3$ and $\exp((\log \log N)^{1/c_k})\leq w\leq \exp((\log N)^{1/10})$. Also let $W=\prod_{p\leq w}p$. Then we have
\begin{align}\label{eq:gowers}
\|\Lambda-\widetilde{\Lambda}_{W}\|_{U^k[N]}\ll_k \exp(-(\log w)^{c_k}).     
\end{align}  
\end{proposition}

\begin{proof}
 For $w=\exp((\log N)^{1/10})$, this is Leng's~\cite[Theorem 6]{Leng-newII}. The general case follows from the same argument; for the sake of completeness we give some details.    

 Let $C_k$ be a large enough constant, and let $c_k>0$ be sufficiently small in terms of $C_k$. Let 
 $$\delta=\exp(-(\log w)^{c_k}).$$ 

 We have $\delta\leq 1/\log N$ by assumption, and clearly either~\eqref{eq:gowers} holds or
 \begin{align}\label{eq:delta2}
 \|\delta(\Lambda-\widetilde{\Lambda}_{W})\|_{U^k[N]}\geq \delta^2.    \end{align}

 The function $\delta(\Lambda-\widetilde{\Lambda}_W)$ is bounded pointwise by $O(1)$. Hence, we can apply the quasipolynomial inverse theorem for the Gowers norms~\cite{LSS}. Recalling that $C_k$ is large, the quasipolynomial inverse theorem and~\eqref{eq:delta2} imply that there exists a nilmanifold $G/\Gamma$ of degree $\leq k-1$, step $\leq k-1$, dimension $D\leq (\log (1/\delta))^{C_k}$ and complexity $M\leq \exp((\log (1/\delta))^{C_k})$, and a function $F\colon G/\Gamma\to \mathbb{C}$ which has Lipschitz norm at most $1$ (and hence is $1$-bounded) such that 
 \begin{align}\label{eq:start}
  \left|\frac{1}{N}\sum_{n\leq N}(\Lambda(n)-\widetilde{\Lambda}_{W}(n))F(g(n)\Gamma)\right|&\gg \exp\left(-\left(\log \frac{1}{\delta}\right)^{C_k}\right). 
  \end{align}
   Combining~\cite[Lemma A.6]{Leng-newII} (which is a Fourier expansion on the vertical torus) and~\cite[Lemma 2.2]{Leng-newII}, we may assume that $F$ is
a vertical character with frequency $\xi$ satisfying $|\xi|\leq \delta^{-(2D)^{C_k}}$ and that $G/\Gamma$ has a one-dimensional vertical component, at the cost of weakening~\eqref{eq:start} to
\begin{align}\label{eq:lowerb}
  \left|\frac{1}{N}\sum_{n\leq N}(\Lambda(n)-\widetilde{\Lambda}_{W}(n))F(g(n)\Gamma)\right|&\gg \exp\left(-\left(\log \frac{1}{\delta}\right)^{C_k^{3}}\right). 
  \end{align}
  
  From~\cite[Proposition 7.1]{TT-JEMS}, we have the estimate\footnote{There, $w$ was fixed to be $\exp((\log N)^{1/10})$, but this makes no difference in the argument.}
 \begin{align}\label{eq:siegelnil}\begin{split}
  \left|\frac{1}{N}\sum_{n\leq N}(\Lambda_W(n)-\widetilde{\Lambda}_W(n))F(g(n)\Gamma)\right|&\ll M^{(2D)^{C_k}} (\widetilde{q})^{-1/(2D)^{C_k}}.
  \end{split}
  \end{align}
  Hence, if~\eqref{eq:start} holds, we may assume that either we have 
  \begin{align}\label{eq:qbound}
   \widetilde{q}\leq \left(M/\exp\left(\left(\log \frac{1}{\delta}\right)^{C_k^{3}}\right)\right)^{(2D)^{C_k^2}}   
  \end{align}
  or  
  \begin{align}\label{eq:start2}
  \left|\frac{1}{N}\sum_{n\leq N}(\Lambda(n)-\Lambda_{W}(n))F(g(n)\Gamma)\right|&\gg \exp\left(-\left(\log \frac{1}{\delta}\right)^{C_k^3}\right). 
  \end{align}

 Now, applying~\cite[Lemma 5.1]{Leng-newII} (which is a version of Vaughan's identity), we may split $\Lambda-\widetilde{\Lambda}_{W}$ and $\Lambda-\Lambda_W$ into type I, twisted type I and type II sums, where we may assume that~\eqref{eq:qbound} holds if twisted type I sums appear. Applying now~\cite[Propositions 5.1--5.4]{Leng-newII} (which are estimates for type I, twisted type I and type II sums involving nilsequences $F(g(n)\Gamma)$
with a vertical character $F$ and where $G/\Gamma$ has a one-dimensional vertical
component, with the twisted type I sum estimate requiring~\eqref{eq:qbound}), the claim follows unless we have a factorisation
\begin{align*}
g(n)=\epsilon(n)g_1(n)\gamma(n)    
\end{align*}
for all $n\in \mathbb{Z}$, where $g_1$ takes values in a $\exp((\log(1/\delta))^{C_k^{10}})$-rational subgroup of $G$ whose step is $\leq s-1$, $\epsilon$ is $(\exp((\log(1/\delta))^{C_k^{10}}),N)$-smooth and $\gamma$ is $\exp((\log(1/\delta))^{C_k^{10}})$-rational. 

Let $Q$ be the period of $\gamma$. Then by the pigeonhole principle there exists an arithmetic progression $P\subset [N]$ of common difference $Q$ and length $\asymp \exp(-(\log(1/\delta))^{C_k^{10}})N$ such that 
\begin{align*}
 \left|\frac{1}{N}\sum_{n\leq N}(\Lambda(n)-\widetilde{\Lambda}_{W}(n))F(g(n)\Gamma)\right|&\leq  \left|\frac{1}{|P|}\sum_{n\in P}(\Lambda(n)-\widetilde{\Lambda}_{W}(n))F(\epsilon_0g_1(n)\gamma_0\Gamma)\right|\\
 &+O(\exp(-(\log(1/\delta))^{C_k^{10}}))    
\end{align*}
for some $\epsilon_0,\gamma_0\in G$. Factorise  $\gamma_0 = \{\gamma_0\}[\gamma_0]$ where $[\gamma_0] \in \Gamma$ and $|\psi(\{\gamma_0\})| \leq 1/2$, with $\psi$ the Mal'cev coordinate map. Writing $g'(n)=\{\gamma_0\}^{-1}g_1(n)\{\gamma_0\}$ and $F_1=F(\epsilon_0\{\gamma_0\}\cdot)1_P$, we conclude from~\eqref{eq:start} that
\begin{align}\label{eq:end}
  \left|\frac{1}{N}\sum_{n\leq N}(\Lambda(n)-\widetilde{\Lambda}_{W}(n))F_1(g'(n)\Gamma)\right|&\gg \exp\left(-\left(\log \frac{1}{\delta}\right)^{C_k^{10}}\right), 
  \end{align}
  where $F_1$ is defined on a nilmanifold of step at most $s-1$. Iterating the arguments from~\eqref{eq:start} to~\eqref{eq:end} $\leq s$ times (and adjusting the value of $\delta$ each time), we conclude that 
  \begin{align*}
  \left|\frac{1}{N}\sum_{n\leq N}(\Lambda(n)-\widetilde{\Lambda}_{W}(n))1_{P'}(n)\right|&\gg \exp\left(-\left(\log \frac{1}{\delta}\right)^{C_k^{10^s}}\right) 
  \end{align*}
for some arithmetic progression $P'$. But this contradicts Proposition~\ref{prop:landaupage}. Hence the claim~\eqref{eq:gowers} follows.   
\end{proof}

\section{M\"obius correlations along polynomial patterns}

We are now ready to prove Theorem~\ref{thm_general}(1), which immediately implies Theorem~\ref{thm_mobius} by the triangle inequality.

\begin{proof}[Proof of Theorem~\ref{thm_general}(1)] Let $P_1,\ldots, P_k\in \mathbb{Z}[y]$ be polynomials satisfying the assumption of Theorem~\ref{thm_general}(1).

For some unimodular $\theta_j\colon \mathbb{Z}\to \mathbb{C}$, we can write 
\begin{align*}
\mathbb{E}_{m\leq N}\left|\mathbb{E}_{n\leq N^d}\prod_{j=1}^k \mu(n+P_j(m))\right|=\mathbb{E}_{m\leq N}\mathbb{E}_{n\leq N^d}\theta_1(m)\prod_{j=1}^k \mu(n+P_j(m))    
\end{align*}
and
\begin{align*}
\mathbb{E}_{n\leq N^d}\left|\mathbb{E}_{m\leq N}\prod_{j=1}^k \mu(n+P_j(m))\right|=\mathbb{E}_{n\leq N^d}\mathbb{E}_{m\leq N}\theta_2(n)\prod_{j=1}^k \mu(n+P_j(m)).    
\end{align*}

For the first average, choose an integer $a$ larger than the absolute values of the coefficients of $y^d$ in the $P_j$, and substitute $n'=n+am^d$ and $P_j'=P_j-am^d$. Let $L=3(a+1)N^d$. Expanding $1_{[N^d]}(n'-am^d)$ in characters modulo $L$ separates the weights in $m$ and $n'$ on the enlarged rectangle $(m,n')\in [N]\times [(a+1)N^d]$, and the sum of the
absolute values of the Fourier coefficients is $O(\log N)$. All the $P_j'$ have degree $d$,
and $\deg(P_\ell'-P_i')=d$ for $i\neq\ell$. For the second average, we absorb every factor corresponding to a constant $P_j$ into $f_0(n)$. In both cases, we relabel the distinguished polynomial as
the last one.

Let $s\ll_{d,k}1$ be as in Proposition~\ref{prop:GVNT3}. Now, by Proposition~\ref{prop:GVNT3} (with $f_i=\mu$ for $i\geq1$) and the assumption on the collection $\{P_1,\ldots,P_k\}$, it suffices to show that 
\begin{align*}
\|\mu\|_{U^s[M]}\ll_A (\log M)^{-A},    
\end{align*}
This bound is precisely Proposition~\ref{prop:leng}. 
\end{proof}

The rest of the paper is devoted to the proof of Theorem~\ref{thm_general}(2) (and the deduction of Theorem~\ref{thm_mangoldt1} from it).

\section{A local-to-global principle for mean values}

The main result of this section is Proposition~\ref{prop:wsieve} below, giving a factorisation result for mean values of a certain class of arithmetic functions that are close to $1$ on average, which, as a corollary, will be applied to correlations of the function $\Lambda_W$, given in~\eqref{eq:lambdaw}. This can be viewed as a local-to-global principle, since the main term in the asymptotic factorises according to the contribution of individual primes to the arithmetic function.

\begin{proposition}[A local-to-global principle for mean values]\label{prop:wsieve} Let $C\geq 1$ and $N\geq 2$. Let $P$ be a nonconstant polynomial with integer coefficients having no fixed prime divisor. Let $P$ have degree at most $C$ and all coefficients bounded by $N^{C}$ in modulus. Let $f\colon \mathbb{N}^2\to \mathbb{C}$ be a function satisfying the following.
\begin{enumerate}
    \item The function $n\mapsto f(p,n)$ is $p$-periodic for any prime $p$.

    \item We have $f(p,n)=1$ for $p>\exp((\log N)^{1/2})$.

    \item We have $|f(p,n)-1|\leq 1_{p\mid P(n)}+C/p$ for all primes $p\leq \exp((\log N)^{1/2})$ and all $n\in \mathbb{N}$. 
\end{enumerate} 
Then we have
\begin{align}\label{eq:fullproduct}
\frac1N \sum_{1\leq n\leq N \atop P(n)\neq 0}\prod_{p}f(p,n)=\prod_{p}\mathbb{E}_{n\in \mathbb{Z}/p\mathbb{Z}}f(p,n)+O_{C}(\exp(-(\log N)^{1/2}/10).     
\end{align}
\end{proposition}

\begin{proof}
We may assume without loss of generality that $C\geq 2$. Let $w=\exp((\log N)^{1/2})$. Then, by assumption (2), it suffices to show that
\begin{align}\label{eq:truncate}
\frac1N \sum_{1\leq n\leq N \atop P(n)\neq 0}\prod_{p\leq w}f(p,n)=\prod_{p\leq w}\mathbb{E}_{n\in \mathbb{Z}/p\mathbb{Z}}f(p,n)+O_{C}(\exp(-(\log N)^{1/2}/10).     
\end{align}

Write $f(p,n)=1+a(p,n)$. Then, by assumptions (1) and (3), $n\mapsto a(p,n)$ is $p$-periodic for every $n\in \mathbb{N}$ and $|a(p,n)|\leq 1_{p\mid P(n)}+C/p$. Writing $W=\prod_{p\leq w}p$, we now obtain
\begin{align*}
\prod_{p\leq w}f(p,n)=\sum_{d\mid W}A(d,n),    
\end{align*}
where $A(d,n)\coloneqq \prod_{p\mid d}a(p,n)$. The function $n\mapsto A(d,n)$ is multiplicative, and for squarefree $d$ we have the crude estimate
\begin{align*}
|A(d,n)|\leq \prod_{\substack{p\mid d\\p\mid P(n)}}\left(1+\frac{C}{p}\right)\prod_{\substack{p\mid d\\p\nmid P(n)}}\frac{C}{p}\leq C^{\omega(d)}\prod_{\substack{p\mid d\\p\nmid P(n)}}\frac{1}{p}=\frac{C^{\omega(d)}(d,P(n))}{d}.   
\end{align*}

Let $s=2(\log N)^{1/2}/5$. Then we can split
\begin{align*}
\prod_{p\leq w}f(p,n)=\sum_{\substack{d\mid W\\d\leq w^s}}A(d,n)+O\Bigg(\sum_{\substack{d\mid W\\d> w^s}}\frac{C^{\omega(d)}(d,P(n))}{d}\Bigg).   \end{align*}
Averaging over $n\leq N$ with $P(n)\neq0$ gives
\begin{align}\label{eq:intermediate}
\frac1N \sum_{1\leq n\leq N \atop P(n)\neq 0}\prod_{p\leq w}f(p,n)=\frac1N\sum_{\substack{d\mid W\\d\leq w^s}} \sum_{1\leq n\leq N \atop P(n)\neq 0}A(d,n)+O\Bigg(\frac1N \sum_{1\leq n\leq N \atop P(n)\neq 0}\sum_{\substack{d\mid W\\d> w^s}}\frac{C^{\omega(d)}(d,P(n))}{d}\Bigg).   \end{align}
Using Rankin's trick and Mertens's bound, for any $n\in [N]$ with $P(n)\neq 0$ we can estimate
\begin{align}\label{eq:Cd}\begin{split}
\sum_{\substack{d\mid W\\d> w^s}}\frac{C^{\omega(d)}(d,P(n))}{d}&\leq e^{-s}\sum_{d\mid W}\frac{C^{\omega(d)}(d,P(n))}{d^{1-1/(\log w)}}\\
&=e^{-s}\prod_{p\leq w}\left(1+\frac{C\cdot (p,P(n))}{p^{1-1/(\log w)}}\right)\\
&\leq e^{-s}\prod_{p\leq w}\left(1+\frac{Ce\cdot (p,P(n))}{p}\right)\\
&\ll e^{-s}(\log w)^{Ce}(1+Ce)^{\omega(|P(n)|)}\\
&\leq e^{-s}(\log w)^{Ce}\tau(|P(n)|)^{B},
\end{split}
\end{align}
where $B=B_C=(\log(1+Ce))/(\log 2)$. 

An elementary inequality due to Landreau~\cite{landreau} states that there exists $K$, depending only on $C$, such that for all $1\leq n\leq (C+1)N^{2C}$ we have 
\begin{align*}
\tau(n)^B\leq K\sum_{\substack{d\mid n\\d\leq N^{1/10}}}\tau(d)^{K}.    
\end{align*}
Using this and the fact (following from~\cite[Lemma 6.1]{BST}) that the congruence $P(x)\equiv 0\pmod{p^r}$ has at most $\min\{Cp^{r-1},Cp^{r(1-1/C)}\}$ solutions for any prime $p$ and $r\geq 1$, besides, write $\nu(d)=\#\{a \pmod d: P(a)\equiv 0\pmod d\}$, we see that
\begin{align}\label{eq:tau}\begin{split}
\frac1N \sum_{1\leq n\leq N \atop P(n)\neq 0}\tau(|P(n)|)^{B}&\ll_C \sum_{d\leq N^{1/10}}\tau(d)^{K}\frac1N \sum_{1\leq n\leq N \atop P(n)\neq 0}1_{d\mid P(n)}\\
&\ll_C \sum_{d\leq N^{1/10}} \tau(d)^K \nu(d) (\frac1d+\frac1N)\\
&\ll_C \prod_{p\leq N^{1/10}}\left(1+ \sum_{r\geq 1} \frac{(r+1)^K \nu(p^r)}{p^r}\right)\\
&\ll_C \prod_{p\leq N^{1/10}} \Big(1+\frac{O_C(1)}{p}\Big)\ll_C (\log N)^{O_C(1)}   
\end{split}
\end{align}
by a standard bound for moments of the divisor function.

Combining~\eqref{eq:intermediate},~\eqref{eq:Cd},\eqref{eq:tau}, we see that
\begin{align}\label{eq:Adn}
\frac1N \sum_{1\leq n\leq N \atop P(n)\neq 0}\prod_{p\leq w}f(p,n)=\sum_{\substack{d\mid W\\d\leq w^s}}\frac1N \sum_{1\leq n\leq N \atop P(n)\neq 0}A(d,n)+O_C(\exp(-(\log N)^{1/2}/10)).   
\end{align}
Meanwhile, one has
\[
\frac1N \sum_{1\leq n\leq N \atop P(n)\neq 0} A(d,n) =\E_{n\leq N} A(d,n) -\frac1N \sum_{n\leq N\atop P(n)=0} A(d,n);
\]
besides, $n\to A(d,n)$ is $d$-periodic, $|A(d,n)| \leq C^{\omega(d)}$ and $P$ has at most $C$ integer roots. Consequently,
\[
\frac1N \sum_{1\leq n\leq N \atop P(n)\neq 0} A(d,n) = \widetilde{A}(d) +O_c\Big(\frac{d C^{\omega(d)}}{N}\Big),
\]
where $\widetilde{A}(d)\coloneqq \mathbb{E}_{n\in \mathbb{Z}/d\mathbb{Z}}A(d,n)$.
Since $d\leq N^{2/5}$ in the above sum, we have 
\begin{align*}
\mathbb{E}_{n\leq N}A(d,n)=\widetilde{A}(d)+O(N^{-3/5+o(1)}).   
\end{align*}
 From the $d$-periodicity of $n\mapsto A(d,n)$ and the multiplicativity of $d\mapsto A(d,n)$, we see that $\widetilde{A}$ is a multiplicative function. Moreover, since $|A(d,n)|\leq C^{\omega(d)}(d,P(n))/d$, we have 
$$|\widetilde{A}(d)|\leq C^{\omega(d)}\prod_{p\mid d}\left(C+1-C/p\right)\leq \frac{(C+1)^{2\omega(d)}}{d}.$$
We conclude that  
\begin{align}\label{eq:A-tilde}
\frac1N \sum_{1\leq n\leq N \atop P(n)\neq 0}\prod_{p\leq w}f(p,n)=\sum_{d\mid W}\widetilde{A}(d)+O\Bigg(\sum_{\substack{d\mid W\\d> w^s}}\frac{(C+1)^{2\omega(d)}}{d}\Bigg)+O_C(\exp(-(\log N)^{1/2}/10)).    
\end{align}
The sum in the error term can be bounded using e.g.~\cite[Lemma 7.5]{BST}, which yields
\begin{align}\label{eq:Csumbound}
\sum_{\substack{d\mid W\\d> w^s}}\frac{(C+1)^{2\omega(d)}}{d}\ll_C  e^{-s/4}\ll \exp(-(\log N)^{1/2}/10).   
\end{align}
By the multiplicativity of $\widetilde{A}$, the main term on the right-hand side of~\eqref{eq:A-tilde} becomes
\begin{align*}
\prod_{p\leq w}\left(1+\widetilde{A}(p)\right).    
\end{align*}
The claim~\eqref{eq:truncate} now follows by noting that $1+\widetilde{A}(p)=1+\mathbb{E}_{n\in \mathbb{Z}/p\mathbb{Z}}a(p,n)=\mathbb{E}_{n\in \mathbb{Z}/p\mathbb{Z}}f(p,n)$.
\end{proof}

\begin{corollary}[A factorisation lemma for correlations]\label{cor:wsieve}
Let $k\in \mathbb{N}$ and $C\geq 1$. Let $N\geq 3$ and $w_1, \ldots, w_k,Q\leq \exp((\log N)^{1/2}/100)$. Let $W_j=\prod_{p\leq w_j}p$. Let  $a_1,\ldots, a_k\in [-N^C, N^{C}]$ be integers. Then we have 
\begin{align}\label{eq:sievelemma}\begin{split}
 &\mathbb{E}_{n\leq N}\prod_{j=1}^{k}\Lambda_{W_j}(Qn+a_j)\\
 &\quad =\prod_{p}\mathbb{E}_{n\in \mathbb{Z}/p\mathbb{Z}}\prod_{j=1}^{k}\Lambda_{(p,W_j)}(Qn+a_j)+O_{C,k}(\exp(-(\log N)^{1/2}/10)), 
 \end{split}
\end{align}
where $\Lambda_p$ for $p$ prime is as in~\eqref{eq:lambdap} and $\Lambda_1\coloneqq 1$. Moreover, let $P_j\in\Z[n]$ be fixed nonconstant polynomials and $b\in \Z$ with $|b|\leq N^{O(1)}$, then we have
\begin{multline*}
\E_{m\leq N} \prod_{1\leq j \leq k} \Lambda_{W_j}(P_j(Qm+b)+a_j) =\prod_{p\leq \max_j w_j} \E_{m\in \Z/p\Z}\prod_{j:p\leq w_j}\Lambda_p(P_j(Qm+b)+a_j)\\
+O_{P_,1\dots,P_k}(\exp(-(\log N)^{1/2}/10)).	
\end{multline*}

\end{corollary}

\begin{proof}
We may omit indices with $W_j=1$, treating an empty product
directly. Let $c_j=(Q,a_j)$. If $p\mid c_j$ for some prime
$p\leq w_j$, then both sides of~\eqref{eq:sievelemma} vanish.
Thus, we may assume that $p\nmid c_j$ for all primes
$p\leq w_j$.
Fist, we notice by taking $f(p,n)=\prod_{j:p\leq w_j}1_{p\nmid Qn+a_j}$ (and multiplying both sides of~\eqref{eq:sievelemma} by $\prod_p \prod_{j=1}^k \frac{\varphi((p,W_j))}{(p,W_j)}$). Properties (1)--(3) required of $f$ in the proof of Proposition~\ref{prop:wsieve} are immediate (with a sufficiently large $C'\ll_{C,k}1$ in place of $C$ and $P(y)=\prod_{j=1}^k\frac{Qy+a_j}{c_j}$). Here $P$ is primitive, which suffices for the root estimates
used in that proof. Applying that proof with
$w=\exp((\log N)^{1/2}/100)$ and $s=40(\log N)^{1/2}$
gives an error $O_{C,k}(\exp(-(\log N)^{1/2}))$.
Multiplication by
$\prod_{j=1}^k W_j/\varphi(W_j)\ll_k(\log N)^{k/2}$
then gives the required error term. Formula \eqref{eq:sievelemma} follows as noting when $P(y)=0$ then for some $1\leq j\leq k$ such that $Qy+a_j=0$ and therefore $\Lambda_{W_j}(Qy+a_j)=0$. Hence, every excluded summation is thus zero, and the assumption $P(y)\neq 0$ can be removed.

For the second statement, let \(c_j\) be the positive greatest common divisor of the coefficients of \(P_j(Qn+b)+a_j\). If \(p\mid c_j\) for some prime \(p\leq w_j\), then both sides of the desired identity vanish. Thus, we may assume that \(p\nmid c_j\) for all primes \(p\leq w_j\).

Set $P(n)=\prod_{j=1}^{k}\frac{P_j(Qn+b)+a_j}{c_j}$
and
\[
f(p,n)=
\begin{cases}
\displaystyle\prod_{\substack{1\leq j\leq k\\p\leq w_j}}
\Lambda_p(P_j(Qn+b)+a_j),
&\text{if }p\leq \max_j w_j,\\[6pt]
1,&\text{otherwise}.
\end{cases}
\]
Since \(P\) is primitive, the argument in the proof of Proposition~\ref{prop:wsieve} applies, yielding
\[
\begin{aligned}
&\frac1N\sum_{\substack{m\leq N\\P(m)\neq0}}
\prod_{j=1}^{k}\Lambda_{W_j}(P_j(Qm+b)+a_j)\\
=&\prod_p\mathbb E_{m\in\mathbb Z/p\mathbb Z}f(p,m)
+O_{P_1,\ldots,P_k}\bigl(\exp(-(\log N)^{1/2}/10)\bigr).
\end{aligned}
\]
Finally, if \(P(m)=0\), then \(P_j(Qm+b)+a_j=0\) for some \(1\leq j\leq k\), so the corresponding summand vanishes. Hence, the restriction \(P(m)\neq0\) can be removed, completing the proof.
\end{proof}

We end this section by proving a lemma that bounds the moments of a sum that arises as an error term when truncating the definition of $\Lambda_{W}$. This will be needed in Section~\ref{sec:lambda}.

\begin{lemma}\label{le:Eerror}
Let $C\geq 1$. Let $w\geq 2$ and write $W=\prod_{p\leq w}p$. Then for any $N\geq 2$ and $s\geq 1$ we have
\begin{align*}
\sum_{n\leq N} \Bigg|\sum_{\substack{d\mid W\\d>w^s}}\mu(d)1_{d\mid n}\Bigg|^{C}\ll_C Ne^{-s}(\log w)^{2^{C+4}}.  
\end{align*}
\end{lemma}

\begin{proof} 
Let $2\ell$ be the least even integer that is $\geq C$. Then, we have
\begin{align*}
\mathbb{E}_{n\leq N} \Bigg|\sum_{\substack{d\mid W\\d>w^s}}\mu(d)1_{d\mid n}\Bigg|^{C}&\leq \mathbb{E}_{n\leq N} \Bigg|\sum_{\substack{d\mid W\\d>w^s}}\mu(d)1_{d\mid n}\Bigg|^{2\ell}\\
&\leq \sum_{\substack{d_1,\ldots, d_{2\ell}\mid W\\d_1,\ldots, d_{2\ell}>w^{s}}}\mathbb{E}_{n\leq N}\mu^2(d_1)\cdots \mu^2(d_{2\ell})\prod_{j=1}^{2\ell}1_{d_j\mid n}\\
&\leq \sum_{\substack{d_1,\ldots, d_{2\ell}\mid W\\d_1,\ldots, d_{2\ell}>w^{s}}}\frac{\mu^2(d_1)\cdots \mu^2(d_{2\ell})}{[d_1,\ldots, d_{2\ell}]}\\
&\leq e^{-s}\sum_{d_1,\ldots, d_{2\ell}\mid W}\mu^2(d_1)\cdots \mu^2(d_{2\ell})\frac{(d_1\cdots d_{2\ell})^{1/(2\ell\log w)}}{[d_1,\ldots, d_{2\ell}]}
\end{align*}

Writing the previous expression as an Euler product and applying Mertens's theorem, we see that it is
\begin{align*}
&=e^{-s}\prod_{p\leq w}\Bigg(1+\sum_{(a_1,\ldots, a_{2\ell})\in \{0,1\}^{2\ell}\setminus \{\mathbf{0}\}}p^{(a_1+\cdots +a_{2\ell})/(2\ell\log w)-\max_i\{a_i\}}\Bigg)\\
&\ll e^{-s}\prod_{p\leq w}\left(1+(2^{2\ell}-1)\frac{e}{p}\right)\\
&\ll_{\ell} e^{-s}(\log w)^{(2^{2\ell}-1)e}\\
&\ll e^{-s}(\log w)^{2^{C+4}},
\end{align*}
as desired.  
\end{proof}

\section{Prime correlations along polynomial patterns}\label{sec:lambda}

\subsection{Overview of proof}

In this section we prove Theorem~\ref{thm_general}(2) and deduce Theorem~\ref{thm_mangoldt1} from it. We begin with a brief sketch to indicate the main tools and why we need to work with the Siegel model, defined in Definition~\ref{def:Siegelmodel}. 

Suppose for simplicity that we want to estimate
\begin{align}\label{eq:lambda3}
S\coloneqq \frac{1}{N^4}\sum_{m\leq N}\sum_{n\leq N^3}\Lambda(n)\Lambda(n+m^3)\Lambda(n+2m^3)    
\end{align}
up to an $O_A((\log N)^{-A})$ error. Let $w_3=\exp((\log  N)^{\frac{1}{20C^2}})$ for a large constant $C$, and let $W_j=\prod_{p\leq w_j} p$. Writing $\Lambda=\widetilde{\Lambda}_{W_3}+E_3$, we see that
\begin{align*}
 S&=\frac{1}{N^4}\sum_{m\leq N}\sum_{n\leq N^3}\Lambda(n)\Lambda(n+m^3)\widetilde{\Lambda}_{W_3}(n+2m^3)\\
 &\quad +\frac{1}{N^4}\sum_{m\leq N}\sum_{n\leq N^3}\Lambda(n)\Lambda(n+m^3)E_3(n+2m^3)\\
 &\coloneqq S_1+S_2.  
\end{align*}
Combining Propositions~\ref{prop:GVNT3} and~\ref{prop:leng2}, we have $|S_2|\ll_A (\log N)^{-A}$. To estimate $S_1$, we choose some $w_2$ that is moderately large in terms of $w_3$ ($w_2=\exp((\log w_3)^C)$ works) and write $\Lambda=\widetilde{\Lambda}_{W_2}+E_2$ to obtain 
\begin{align*}
S_1&=\frac{1}{N^4}\sum_{m\leq N}\sum_{n\leq N^3}\Lambda(n)\widetilde{\Lambda}_{W_2}(n+m^3)\widetilde{\Lambda}_{W_3}(n+2m^3)\\
&\quad +\frac{1}{N^4}\sum_{m\leq N}\sum_{n\leq N^3}\Lambda(n)E_2(n+m^3)\widetilde{\Lambda}_{W_3}(n+2m^3)\\
&\coloneqq S_3+S_4.    
\end{align*}

Using Lemma~\ref{le:Eerror}, we see that $\widetilde{\Lambda}_{W_3}(n)$ is up to negligible error (in the $L^1$ norm) equal to 
$$(1-\widetilde{\chi}(n)n^{\widetilde{\beta}-1})\frac{W_3}{\varphi(W_3)}\sum_{\substack{t\mid W_3\\ t\leq w_3^{(\log \log N)^2}}}\mu(t)1_{t\mid n}.$$ 
Substituting this into $S_4$ and applying the triangle inequality, and denoting by $\widetilde{q}\leq w_3$ the modulus of $\widetilde{\chi}$, we see that for some $1\leq a,t\leq \widetilde{q}w_3^{(\log \log N)^2}$ we have
\begin{align*}
|S_4|\ll \frac{W_3}{\varphi(W_3)}\frac{\widetilde{q}w_3^{(\log \log N)^2}}{N^4}\left|\sum_{m\leq N}\sum_{n\leq N^3}\Lambda(n)E_2(n+m^3)1_{n+2m^3\equiv a\pmod t}\right|.
\end{align*}
Splitting the $m$ variable into progressions modulo $t$ and applying Propositions~\ref{prop:GVNT3} and~\ref{prop:leng2}\footnote{Note that we are applying Propositions~\ref{prop:GVNT3} with the coefficients of the polynomials involved bounded by $O(t^{3})$. Hence, Proposition~\ref{prop:GVNT3} involves a loss of $t^{O(1)}$, but we win back this loss thanks to the strong bound on the Gowers norm of $E_2$.}, we conclude that $|S_4|\ll_A (\log N)^{-A}$. Performing again a similar splitting $\Lambda=\widetilde{\Lambda}_{W_1}+E_1$ with $w_1$ large in terms of $w_2$ ($w_1=\exp((\log w_2)^C)$ works) and arguing as above, we conclude that
\begin{align}\label{eq:S}
S=\frac{1}{N^4}\sum_{m\leq N}\sum_{n\leq N^3}\widetilde{\Lambda}_{W_1}(n)\widetilde{\Lambda}_{W_2}(n+m^3)\widetilde{\Lambda}_{W_3}(n+2m^3)+O_A((\log N)^{-A}).    
\end{align}
Now the remaining task is to evaluate the correlations of the Siegel model with error terms that save an arbitrary power of logarithm. Writing $\widetilde{\Lambda}_{W_j}=\Lambda_{W_j}+E_j'$, the right-hand side of~\eqref{eq:S} splits into a sum of $8$ averages. The average involving each of the functions $\Lambda_{W_1}, \Lambda_{W_2}, \Lambda_{W_3}$ can be evaluated using Corollary~\ref{cor:wsieve}, and it gives us our main term. All the other averages involving at least one copy of $E_j'$ involve the correlations of a Dirichlet character of large conductor, and (after some work) they can be shown to be small by using the Weil bound.

If we tried to run the same argument with the simpler model $\Lambda_W$, we would run into serious trouble, since we can only save an arbitrary power of logarithm in the Gowers norm $\|\Lambda-\Lambda_{W_j}\|_{U^s[N]}$, whereas we would need to save at least $w_{j+1}^{C(\log \log N)^2}$, which is bigger since $w_3$ needs to be at least a large power of logarithm. The use of the Siegel model rectifies this, since we have a much better quasipolynomial bound on $\|\Lambda-\widetilde{\Lambda}_{W_j}\|_{U^k[N]}$ as a function of $w_j$.   

\subsection{Correlations of the Siegel model}

In what follows, we will need the following simple lemma.

\begin{lemma}\label{le:beta} Let $k\in \mathbb{N}$ be fixed, and let $P_1,...,P_k\in \mathbb{Z}[y]$ be fixed polynomials. Let $\beta_p(m), \beta_p'(m)$ be as in~\eqref{eq:betapm}. For any $A, N\geq 2$ and $m\in [C_0,N]$, where $C_0$ is chosen so that $P_i(m)\neq P_j(m)$ for all $m\geq C_0$, we have 
\begin{align*}
\prod_{p>(\log N)^{2A}}\beta_p(m)=1+O_A((\log N)^{-A}).
\end{align*}
\end{lemma}

\begin{proof}
Note first  that 
\begin{align*}
\left(\frac{p}{p-1}\right)^k\left(1-\frac{k}{p}\right)\leq \beta_p(m)\leq \begin{cases}
\left(\frac{p}{p-1}\right)^k\left(1-\frac{k}{p}\right),\quad &p\nmid \prod_{1\leq i<j\leq k}(P_i(m)-P_j(m))\\
\left(\frac{p}{p-1}\right)^{k},\quad &p\mid \prod_{1\leq i<j\leq k}(P_i(m)-P_j(m)).
\end{cases} 
\end{align*}
Since $\prod_{p>(\log N)^{2A}}(\frac{p}{p-1})^k(1-\frac{k}{p})=1+O_A((\log N)^{-2A})$, the claim follows by observing that for any $m\in [C_0,N]$ and $1\leq i<j\leq k$ we have 
\begin{align*}
\prod_{\substack{p\mid P_i(m)-P_j(m)\\p>(\log N)^{2A}}}\left(\frac{p}{p-1}\right)&\leq \exp\left(\sum_{\substack{p\mid P_i(m)-P_j(m)\\p>(\log N)^{2A}}}\frac{1}{p-1}\right)\leq \exp\left(\frac{2\log(|P_i(m)-P_j(m)|+2)}{(\log N)^{2A}-1}\right)\\
&= 1+O_A((\log N)^{-A}).    
\end{align*}        
\end{proof}

We proceed to calculate the correlations of the Siegel model $\widetilde{\Lambda}_W$ (given in Definition~\eqref{def:Siegelmodel}) along polynomial progressions; this will give the main term in Theorem~\ref{thm_general}(2).

\begin{lemma}[Siegel model correlations]\label{le:main}
Let $k,d\in \mathbb{N}$ be fixed, and let $P_1,\ldots, P_k\in \mathbb{Z}[y]$ be fixed pairwise distinct polynomials of degree at most $d$. Let $C\geq 3$ be fixed, let $N,A\geq 3$ and $\exp((\log N)^{\frac{1}{C^{2k}}})\leq w_1,\ldots, w_k\leq \exp((\log N)^{1/2}/(100k))$. Let $W_j=\prod_{p\leq w_j}p$. 
\begin{enumerate}
    \item We have 
    \begin{align*}
\mathbb{E}_{n\leq N^d}\prod_{j=1}^k \widetilde{\Lambda}_{W_j}(n+P_j(m))=\prod_{p}\beta_p(m)+O_A((\log N)^{-A})    
\end{align*}
for all but $\ll_A N/(\log N)^{A}$ integers $m\in [C_0,N]$, where $C_0>0$ is sufficiently large in terms of $A$ and $P_1,\dots,P_k$.

\item Suppose additionally that all the $P_j$ are nonconstant. There exists some constant $B=B(P_1,\dots,P_k,A)$ that is sufficiently large and such that 
\begin{align*}
\mathbb{E}_{m\leq N}\prod_{j=1}^k \widetilde{\Lambda}_{W_j}(n+P_j(m))=\prod_{p\leq (\log N)^B}\beta_p'(n)+O_A((\log N)^{-A})    
\end{align*}
for all but $\ll_A N^d/(\log N)^{A}$ integers $n\in [N^d]$. 
\end{enumerate}
\end{lemma}

\begin{proof}\textbf{Proof of part (1).}
By definition, we can decompose
\begin{align*}
\widetilde{\Lambda}_{W_j}(n)=\Lambda_{W_j}(n)-\Lambda_{W_j}(n)|n|^{\widetilde{\beta}_j-1}\widetilde{\chi}_j(|n|).    
\end{align*}
Substituting this into the left-hand side of the first statement, we see that it suffices to show that
\begin{align}\label{eq:goal1a}
\mathbb{E}_{n\leq N^d}\prod_{j=1}^k \Lambda_{W_j}(n+P_j(m))=\prod_{p}\beta_p(m)+O_A((\log N)^{-A})    
\end{align}
for all $m\in [C_0,N]$, and that, for any nonempty set $\mathcal{J}\subset [k]$, we have
\begin{align}\label{eq:goal1b}
\mathbb{E}_{n\leq N^d}\prod_{j=1}^{k} \Lambda_{W_j}(n+P_j(m))\prod_{j\in \mathcal{J}}|n+P_j(m)|^{\widetilde{\beta}_j-1}\widetilde{\chi}_j(|n+P_j(m)|)\ll_A(\log N)^{-A}    
\end{align}
for all but $\ll_AN/(\log N)^{A}$ integers $m\in [C_0,N]$. Note that~\eqref{eq:goal1a} follows directly from Corollary~\ref{cor:wsieve} in view of Lemma~\ref{le:beta}.

We are left with showing~\eqref{eq:goal1b}. If $\widetilde{\chi}_j=0$ for some $j\in\mathcal{J}$,
then~\eqref{eq:goal1b} is immediate, so we may assume that
these characters are all nonzero. Using the identity
\begin{align}\label{eq:integral}
y^{\widetilde{\beta}-1}=1+\int_{1}^{CN^d}(\widetilde{\beta}-1)u^{\widetilde{\beta}-2}1_{u\leq y}\d u
\end{align}
for $1\leq y\leq CN^d$ and exchanging the order of integration and summation, we reduce to showing that for all but $\ll_A N/(\log N)^{A}$ choices of $m\in [C_0,N]$ we have
\begin{align*}
&\mathbb{E}_{n\leq N^d}\prod_{j=1}^{k} \Lambda_{W_j}(n+P_j(m))\prod_{j\in \mathcal{J}}1_{|n+P_j(m)|\geq u_j}\widetilde{\chi}_j(|n+P_j(m)|)\ll_A(\log N)^{-A} 
\end{align*}
for all $u_j\in [1,N^{d}]$. Since the condition $|n+P_j(m)|\geq u_j$ is equivalent to $n$ belonging to a union of two intervals and since $\widetilde{\chi}_j(-n)=\widetilde{\chi}_j(-1)\widetilde{\chi}_j(n)$, we reduce to showing that
for all but $\ll_A N/(\log N)^{A}$ choices of $m\in [C_0,N]$ we have
\begin{align}\label{eq:goal1bb}
&\mathbb{E}_{n\leq t}\prod_{j=1}^{k} \Lambda_{W_j}(n+P_j(m))\prod_{j\in \mathcal{J}}\widetilde{\chi}_j(n+P_j(m))\ll_A(\log N)^{-A} 
\end{align}
for all $t\in [N^d/(\log N)^{A},N^d]$. 

Let 
$$Q=\operatorname{lcm}_{j\in\mathcal{J}}\widetilde{q}_j,$$
where we recall that $\widetilde{q}_j\leq w_j\leq \exp((\log N)^{1/2}/(100k))$ is the modulus of the character $\widetilde{\chi}_j$. Then
\begin{align}\label{eq:Qsize}
(\log N)^{A}\ll_A  Q\leq \exp((\log N)^{1/2}/100),   
\end{align}
with the lower bound coming from Siegel's bound. 
 Splitting the $n$ variable in~\eqref{eq:goal1bb} into progressions $\pmod Q$ it suffices to show that for all $1\leq a\leq Q$ and all but $\ll_AN/(\log N)^{A}$ integers $m\in [C_0,N]$ we have
\begin{align*}
\mathbb{E}_{a\leq Q}\prod_{j\in \mathcal{J}}\widetilde{\chi}_j(a+P_j(m))\mathbb{E}_{n\leq t/Q}\prod_{j=1}^k \Lambda_{W_j}(Qn+a+P_j(m))\ll_A(\log N)^{-A}   
\end{align*}
for all $t\in [N^d/(\log N)^{A},N^d]$.
By Corollary~\ref{cor:wsieve}, the inner average over $n$ here is
\begin{align}\label{eq:pdiv}
\prod_{p}\mathbb{E}_{n\in \mathbb{Z}/p\mathbb{Z}}\prod_{j=1}^{k}\Lambda_{(p,W_j)}(Qn+a+P_j(m))+O_{d,k}(\exp(-(\log N)^{1/2})/10).    
\end{align}
If $p\nmid Q$, the factor depending on $p$ here is by a change of variables equal to 
\begin{align*}
  \mathbb{E}_{n\in \mathbb{Z}/p\mathbb{Z}}\prod_{j=1}^{k}\Lambda_{(p,W_j)}(n+P_j(m)),
\end{align*}
and if instead $p\mid Q$ it is equal to $$\prod_{j=1}^k\Lambda_{(p,W_j)}(a+P_j(m)).$$
Now, separating the contributions of $p\mid Q$ and $p\nmid Q$ in~\eqref{eq:pdiv}, it suffices to show that
\begin{align}\label{eq:Weil}
\mathbb{E}_{a\leq Q}\prod_{j\in \mathcal{J}}\widetilde{\chi}_j(a+P_j(m))\prod_{p\mid Q}\prod_{\substack{1\leq j\leq k\\p\leq w_j}} 1_{a+P_j(m)\not \equiv 0\pmod p}\ll_A (\log N)^{-A}   
\end{align}
for all but $\ll_AN/(\log N)^{A}$ integers $m\in [C_0,N]$.

For any $j\in \mathcal{J}$ and any prime power $p^r$, write $\widetilde{\chi}_{j,p^r}$ for the character $\pmod{(p^r,\widetilde{q}_j)}$ induced by $\widetilde{\chi}_j\pmod {\widetilde{q}_j}$.
Then, by the Chinese remainder theorem,~\eqref{eq:Weil} factorises as
\begin{align}\label{eq:CRT}
\prod_{p^{r}\mid \mid Q}\mathbb{E}_{a\in \mathbb{Z}/p^{r}\mathbb{Z}}\prod_{j\in \mathcal{J}}\widetilde{\chi}_{j,p^r}(a+P_j(m))\prod_{\substack{1\leq j\leq k\\p\leq w_j}} 1_{a+P_j(m)\not \equiv 0\pmod p}.     
\end{align}
 Since $\widetilde{\chi}_j$ are primitive real characters, their moduli are of the form $2^{b_j}q_j'$ with $q_j'$ squarefree and $b_j\leq 3$. Hence, for any prime $p>2$ the average in~\eqref{eq:CRT} equals to
\begin{align}\label{eq:aPj}
\mathbb{E}_{a\in \mathbb{Z}/p\mathbb{Z}}\prod_{j\in \mathcal{J}}\widetilde{\chi}_{j,p}(a+P_j(m))+O_{d,k}(1/p).    
\end{align}

From the Weil bound~\cite[Corollary 11.24]{iw-kow}, the average in~\eqref{eq:aPj} is $\leq K p^{-1/2}$ in modulus for some constant $K=K_{d,k}$ unless there exist $i\neq j$ such that $P_i(m)-P_j(m)\equiv 0\pmod p$. Hence, for some $1\leq i<j\leq k$,  the expression~\eqref{eq:CRT} is
\begin{align}\label{eq:Q-1}
\ll \left(\prod_{p\mid Q}Kp^{-1/2}\right) \left(\prod_{\substack{p\mid P_i(m)-P_j(m)\\p\mid Q}}p^{1/2}\right)^{k^2}\ll_{d,k} Q^{-1/2+o(1)} \left(\prod_{\substack{p\mid P_i(m)-P_j(m)\\p\mid Q}}p^{1/2}\right)^{k^2},  
\end{align}
since $Q/2^b$ is squarefree for some $b\leq 3$. From~\eqref{eq:Qsize}, we see that $Q^{-1/10}\ll_A(\log N)^{-A}$, so it suffices to show that for any nonzero polynomial $R\in \mathbb{Z}[y]$ for all but $\ll_{A,R} N/(\log N)^{A}$ integers $m\in [C_0,N]$ we have
\begin{align}\label{eq:largegcd}
\prod_{\substack{p\mid R(m)\\p\mid Q}}p\leq Q^{1/(3k^2)}.    
\end{align}
Note that, by Lagrange's theorem, for any polynomial $R\in \mathbb{Z}[y]$ of degree $d\in \mathbb{N}$, the congruence $R(m)\equiv 0\pmod p$ has $\leq d$ solutions for all primes $p$ that are large enough in terms of $R$. Now, by Markov's inequality, the number of $m\in [C_0,N]$ violating~\eqref{eq:largegcd} is
\begin{align*}
&\leq Q^{-1/(3k^2)}\sum_{m\leq N}(R(m),Q)\\
&\leq Q^{-1/(3k^2)}\sum_{t\mid Q}t\sum_{\substack{m\leq N\\t\mid R(m)}}1\\
&\ll_R Q^{-1/(3k^2)}\sum_{t\mid Q}d^{\Omega(t)}N\\
&\leq Q^{-1/(3k^2)}\tau(Q)d^{\Omega(Q)}N\\
&\ll Q^{-1/(3k^2)+o(1)}N,
\end{align*}
since $\Omega(Q)\leq \omega(Q)+O(1)\ll (\log Q)/(\log \log Q)$ by the fact that $Q/2^b$ is squarefree for some $b\leq 3$. By~\eqref{eq:Qsize}, we conclude that the number of $m\in [N]$ not satisfying~\eqref{eq:largegcd} is $\ll_{A,R} N/(\log N)^{A}$. Hence these integers $m\in [C_0,N]$, may be included in the exceptional set of $m$. 

\textbf{Proof of part (2)} The proof proceeds the same way as for Lemma~\ref{le:main}(1), up to the truncated singular series, which we claimed as
\begin{align}\label{eq:cliam-truncate-series}
\E_{n\leq N^d}\Big|	\E_{m\leq N} \prod_j \Lambda_{W_j}(n+P_j(m)) - \prod_{p\leq (\log N)^B}\beta_p'(n)\Big|\ll(\log N)^{-A},
\end{align}
and  a swapping of the averaging variable, until (in analogy with~\eqref{eq:Weil}) we we are left with showing that
\begin{align}\label{eq:Weil2}
\mathbb{E}_{b\leq Q}\prod_{j\in \mathcal{J}}\widetilde{\chi}_j(n+P_j(b))\prod_{p\mid Q}\prod_{\substack{1\leq j\leq k\\p\leq w_j}} 1_{n+P_j(b)\not \equiv 0\pmod p}\ll_A (\log N)^{-A}   
\end{align}
for all but $\ll_AN^d/(\log N)^{A}$ integers $n\in [N^d]$.

We shall prove \eqref{eq:cliam-truncate-series} first. For convenience, set $w=\max_j w_j$.
 Applying Corollary~\ref{cor:wsieve}
with $Q=1$, $b=0$ and $a_j=n$, we obtain, uniformly for $n\leq N^d$,
\[
\E_{m\leq N}\prod_j \Lambda_{W_j}(n+P_j(m)) =\prod_{p\leq w}\gamma_p(n)+O_{P_1,\dots,P_k}(\exp(-\sqrt{\log N}/10)),
\]
where $\gamma_p(n)= \E_{m\in \Z/p\Z} \prod_{j: p\leq w_j} \Lambda_p(n+P_j(m))$. Since $w_j\geq (\log N)^B$ for all $1\leq j\leq k$, we have $\gamma_p(n)=\beta'_p(n)$ whenever $p\leq (\log N)^B$. Therefore, it suffices
to prove that
\begin{align}\label{eq:aaa}
\E_{n\leq N^d}
\big|\prod_{p\leq (\log N)^B}\beta'_p(n)\big|
\Big|\prod_{(\log N)^B<p\leq w}\gamma_p(n)-1\Big|
\ll(\log N)^{-A}.
\end{align}

For $(\log N)^B<p\leq w$, write $r_p=|\{j:p\leq w_j\}|$.
Since the $P_j$ are nonconstant, for sufficiently large $N$,
the root-counting bound gives, uniformly in $n$,
\[
\left(\frac{p}{p-1}\right)^{r_p}
\left(1-\frac{\sum_{j:p\leq w_j}\deg P_j}{p}\right)
\leq\gamma_p(n)
\leq
\left(\frac{p}{p-1}\right)^{r_p}.
\]
In particular, $\gamma_p(n)=1+O(p^{-1})$. Moreover, each equation $a+P_j(v)=0$ has exactly $p$ solutions
$(a,v)\in(\Z/p\Z)^2$, whereas two distinct such equations have
$O(1)$ common solutions. Thus, inclusion--exclusion yields
\[
\E_{a\in\Z/p\Z}\gamma_p(a)
=
\left(\frac{p}{p-1}\right)^{r_p}
\left(1-\frac{r_p}{p}+O(p^{-2})\right)
=
1+O(p^{-2}).
\]
Besides, we also have
\[
\E_{a\in\Z/p\Z}\gamma_p(a)^2 = 1+2 \E_{a\in\Z/p\Z}(\gamma_p(a)-1) + \E_{a\in\Z/p\Z}(\gamma_p(a)-1)^2 =1+O(p^{-2}).
\]
Hence, the application of Proposition~\ref{prop:wsieve} to both $\gamma_p$ and $\gamma_p^2$ shows that
\[
\E_{n\leq N^d}\prod_{(\log N)^B<p\leq w}\gamma_p(n)
= \prod_{(\log N)^B<p\leq w}\E_{a\in\Z/p\Z}\gamma_p(a) +O(\exp(-\sqrt{\log N}/10)) =1+O((\log N)^{-B}),
\]
and similarly, 
\[
\E_{n\leq N^d}\big(\prod_{(\log N)^B<p\leq w}\gamma_p(n)\big)^2
= 1+O((\log N)^{-B}).
\]
Consequently,
\[
\E_{n\leq N^d} \big|\prod_{(\log N)^B <p\leq w} \gamma_p(n) -1\big|^2 = \E_{n} \big(\prod_{(\log N)^B<p\leq w}\gamma_p(n)\big)^2 -2\E_{n}\prod_{(\log N)^B<p\leq w}\gamma_p(n) +1 \ll (\log N)^{-B}.
\]
Thus,  \eqref{eq:aaa} follows by noting Mertens's bound gives
\[
0\leq\prod_{p\leq (\log N)^B}\beta'_p(n)
\leq
\prod_{p\leq (\log N)^B}\left(\frac{p}{p-1}\right)^k
\ll(\log \log N)^k,
\]
and provided that $B>3A$.

Let's prove \eqref{eq:Weil2} now. As in the proof of part (1), for any $j\in \mathcal{J}$ and any prime power $p^r$, write $\widetilde{\chi}_{j,p^r}$ for the character $\pmod{(p^r,\widetilde{q}_j)}$ induced by $\widetilde{\chi}_j\pmod {\widetilde{q}_j}$.
Then, by the Chinese remainder theorem,~\eqref{eq:Weil} factorises as
\begin{align}\label{eq:CRT2}
\prod_{p^{r}\mid \mid Q}\mathbb{E}_{b\in \mathbb{Z}/p^{r}\mathbb{Z}}\prod_{j\in \mathcal{J}}\widetilde{\chi}_{j,p^r}(n+P_j(b))\prod_{\substack{1\leq j\leq k\\p\leq w_j}} 1_{n+P_j(b)\not \equiv 0\pmod p}.     
\end{align}
 Since $\widetilde{\chi}_j$ are primitive real characters, their moduli are of the form $2^{b_j}q_j'$ with $q_j'$ squarefree and $b_j\leq 3$. Hence, for any prime $p>2$ the average in~\eqref{eq:CRT2} equals to
\begin{align}\label{eq:bPj}
\mathbb{E}_{b\in \mathbb{Z}/p\mathbb{Z}}\prod_{j\in \mathcal{J}}\widetilde{\chi}_{j,p}(n+P_j(b))+O_{d,k}(1/p).    
\end{align}
By the Weil bound, this is $\leq K p^{-1/2}$ in modulus for some constant $K=K_{d,k}$ unless the polynomial
$\prod_{\substack{j\in\mathcal{J}\\p\mid\widetilde{q}_j}}(P_j(y)+n)$ is of the form $cg(y)^2$ modulo $p$ for some $c\in \mathbb{Z}/p\mathbb{Z}$ and some polynomial $g$. This can only happen if $p$ divides the product $\Delta(n)$ of the leading coefficient and the
discriminant of $f_n(y)\coloneqq (P_1(y)+n)\cdots (P_k(y)+n)$. The expression $\Delta$ is some polynomial with integer coefficients. Hence,~\eqref{eq:CRT2} becomes
\begin{align*}
\ll \left(\prod_{p\mid Q}Kp^{-1/2}\right)\left(\prod_{\substack{p\mid \Delta(n)\\p\mid Q}}p^{1/2}\right)\ll_{d,k} Q^{-1/2+o(1)} \left(\prod_{\substack{p\mid \Delta(n)\\p\mid Q}}p^{1/2}\right).  
\end{align*}
Now we may use~\eqref{eq:largegcd} to conclude that this is $\ll Q^{-1/10}\ll_A (\log N)^{-A}$ for all but $\ll_A N^{d}/(\log N)^{A}$ integers $n\in [N^d]$. 
\end{proof}

\subsection{Proof of Theorem~\ref{thm_general}(2)}

We are now ready to prove Theorem~\ref{thm_general}(2) and to deduce Theorem~\ref{thm_mangoldt1} from it. We begin with the latter task.

\begin{proof}[Proof of Theorem~\ref{thm_mangoldt1} assuming Theorem~\ref{thm_general}(2)]
Let $a_j$ be the coefficient of $m^d$ in $P_j$, allowing $a_j=0$. Set $a=2+\max_j|a_j|$ and put $Q_j(m)=P_j(m)+am^d$. Then from the assumption, we have $\deg Q_i=d$ for all $1\leq i\leq k$ and $\deg(Q_i-Q_j)=d$ for all $1\leq i<j\leq k$.

The contribution of $m<C_0$ to the original average is $O_{P_1,\ldots,P_k}(N^{-1}(\log N)^k)$. Hence, by Theorem~\ref{thm_general}(2) and the triangle inequality, it suffices to prove that
\begin{align*}
N^{-1}\sum_{C_0\leq m\leq N}\prod_p\beta_p(m)
=\prod_p\beta_p+O_A((\log N)^{-A}).
\end{align*}
Set $w=\exp((\log N)^{1/2})$. Since the argument in the proof of Lemma~\ref{le:beta} gives, uniformly for $C_0\leq m\leq N$,
\begin{align*}
\prod_{p>w}\beta_p(m)
=1+O_{P_1,\ldots,P_k}\left(\frac{\log N}{w}\right);
\end{align*}
and $\beta_p=1+O_{P_1,\ldots,P_k}(p^{-2})$, we further reduce to  establish
\begin{align*}
N^{-1}\sum_{C_0\leq m\leq N}\prod_{p\leq w}\beta_p(m)
=\prod_{p\leq w}\mathbb{E}_{m\in\mathbb{Z}/p\mathbb{Z}}\beta_p(m)
+O_A((\log N)^{-A}).
\end{align*}
In fact, this estimates follows from Proposition~\ref{prop:wsieve} with $f(p,n)=1+(\beta_p(n)-1)1_{p\leq w}$.
\end{proof}

\begin{proof}[Proof of Theorem~\ref{thm_general}(2)]
Shifting the collection $\{P_1,\ldots, P_k\}$ by a constant if necessary, we
may assume that none of the $P_i$ is the zero polynomial.  We may  first assume  that $\deg P_i=d$ for all $1\leq i\leq k$.

For $j\in [k]$, let
\begin{align}\label{eq:wjdef}
w_j\coloneqq \exp(( \log N)^{\frac{1}{20C^{j-1}}}),
\end{align}
where $C=C(d,k)$ is a large enough constant. Then 
\begin{align*}
w_{j+1}\leq\exp((\log w_j)^{1/C}),    
\end{align*}
which will eventually allow us to apply Proposition~\ref{prop:leng2}.

For some 1-bounded functions $\theta_1,\theta_2\colon \mathbb{Z}\to \mathbb{C}$, we can write
\begin{align}\label{eq:Lambdaaverage}\begin{split}
&N^{-1} \sum_{C_0\leq m\leq N}\left|\mathbb{E}_{n\leq N^d}\prod_{j=1}^k \Lambda(n+P_j(m))-\prod_p\beta_p(m)\right|\\
&\quad = N^{-1} \sum_{C_0\leq m\leq N}\theta_1(m)\left(\mathbb{E}_{n\leq N^d}\prod_{j=1}^k \Lambda(n+P_j(m))-\prod_p\beta_p(m)\right)
\end{split}
\end{align}
and
\begin{align}\label{eq:Lambdaaverage2}
&\mathbb{E}_{n\leq N^d}\left|\mathbb{E}_{m\leq N}\prod_{j=1}^k \Lambda(n+P_j(m))-\prod_{p\leq (\log N)^B}\beta_p'(n)\right|\nonumber\\
&\quad = \mathbb{E}_{n\leq N^d}\theta_2(n)\left(\mathbb{E}_{m\leq N}\prod_{j=1}^k \Lambda(n+P_j(m))-\prod_{p\leq (\log N)^B}\beta'_p(n)\right).    
\end{align}
From Lemma~\ref{le:beta} and the estimates $\beta_p(m),\beta_p'(m)=1+O_{k}(1/p)$, we see that the expressions inside the brackets here are $\ll (\log N)^{O_k(1)}$.

We split the average on the right-hand side of~\eqref{eq:Lambdaaverage},~\eqref{eq:Lambdaaverage2}
into a sum of $2^k$ different averages by substituting $$\Lambda=\widetilde{\Lambda}_{W_j}+E_j$$
in the $j$th factor in~\eqref{eq:Lambdaaverage}, where $\widetilde{\Lambda}_{W_j}$ is as in Definition~\ref{def:Siegelmodel}. Note for later use that
\begin{align}\label{eq:Ebound}
|E_j(n)|\ll \log |n|+ \log N.    
\end{align}
Then it suffices to show that
\begin{align}\label{eq:goal1}
\mathbb{E}_{n\leq N^d}\prod_{j=1}^k \widetilde{\Lambda}_{W_j}(n+P_j(m))=\prod_{p}\beta_p(m)+O_A((\log N)^{-A})    
\end{align}
for all but $\ll_AN/(\log N)^{A}$ integers $m\in [C_0,N]$, and that
\begin{align}\label{eq:goal1c}
\mathbb{E}_{m\leq N}\prod_{j=1}^k \widetilde{\Lambda}_{W_j}(n+P_j(m))=\prod_{p\leq (\log N)^B}\beta_p'(n)+O_A((\log N)^{-A})    
\end{align}
for all but $\ll_A N^d/(\log N)^{A}$ integers $n\in [N^d]$,
and that for any $0\leq \ell\leq k-1$ and any unimodular $\theta_1,\theta_2\colon \mathbb{Z}\to \mathbb{C}$ we have
\begin{align}\label{eq:goal2}\begin{split}
&\mathbb{E}_{n\leq N^d}\mathbb{E}_{m\leq N}\theta_1(m)\theta_2(n)\prod_{j=1}^{k-\ell-1} \Lambda(n+P_j(m))E_{k-\ell}(n+P_{k-\ell}(m))\prod_{j=k-\ell+1}^{k}\widetilde{\Lambda}_{W_j}(n+P_j(m))\\
&\quad \ll_A(\log N)^{-A}. 
\end{split}
\end{align}
The estimates~\eqref{eq:goal1},~\eqref{eq:goal1c} follow directly from Lemma~\ref{le:main}. Hence, the remaining task is proving~\eqref{eq:goal2}.

Suppose first that $\ell=0$. In this case, the claim~\eqref{eq:goal2} simplifies as
\begin{align*}
\mathbb{E}_{n\leq N^d} \mathbb{E}_{m\leq N}\theta_1(m)\theta_2(n)\prod_{j=1}^{k-1}\Lambda(n+P_j(m))\cdot E_{k}(n+P_k(m))\ll_A(\log N)^{-A}.  
\end{align*}
By Proposition~\ref{prop:GVNT3} (and the bounds $\Lambda(n)\ll \log |n|$ and~\eqref{eq:Ebound}), this holds unless
\begin{align*}
\|E_k\|_{U^s[CN^d]}\gg (\log N)^{-O_{A,d}(1)}  
\end{align*}
for some $1\leq s,C\ll_{P_1,\ldots, P_k} 1$. But since $E_{k}=\Lambda-\widetilde{\Lambda}_{W_k}$, this contradicts Proposition~\ref{prop:leng2} (recalling~\eqref{eq:wjdef} and that $C=C(d,k)$ is large enough there). We may suppose from now on that $1\leq \ell\leq k-1$.

 Let $$r=(\log \log N)^2.$$
 For $W,V\geq 1$, define the truncated function
\begin{align}\label{eq:VW}
\Lambda_{W,\leq V}(n)\coloneqq \frac{W}{\varphi(W)}\sum_{\substack{d\mid W\\d\leq V}}\mu(d)1_{d\mid n};  
\end{align}
then by the M\"obius inversion formula we have $\Lambda_W(n)=\Lambda_{W,\leq V}(n)$ for any $V\geq n$. Write
\begin{align}\label{eq:lambdadecompose}
\widetilde{\Lambda}_{W_j}(n)=\Lambda_{W_j,\leq w_j^{r}}(n)\cdot (1-|n|^{\widetilde{\beta}-1}\widetilde{\chi}(|n|))+\mathcal{E}_j(n).    
\end{align}
Now split~\eqref{eq:goal2} into a sum of $2^{\ell}$ different averages by writing~\eqref{eq:lambdadecompose} in the $j$th factor of~\eqref{eq:goal2} for $k-\ell+1\leq j\leq k$. Using the bounds $|\Lambda_{W_j}(n)|, |E_j(n)|\ll_{C,d} \log N$ for $|n|\leq CN^d$ and 
\begin{align*}
\mathbb{E}_{n\leq CN^d}|\Lambda_{W_j,\leq w_j^r}(n)|^k\leq \left(\frac{W_j}{\varphi(W_j)}\right)^k\mathbb{E}_{n\leq CN^d}\tau(n)^k\ll_C (\log N)^{2^k-1+\frac{k}{20C^{j-1}}}    
\end{align*}
and H\"older's inequality, it suffices to show that
\begin{align}\label{eq:goal2b}\begin{split}
&\mathbb{E}_{n\leq N^d}\mathbb{E}_{m\leq N}\theta_1(m)\theta_2(n)\prod_{j=1}^{k-\ell-1}\Lambda(n+P_{j}(m)) \cdot E_{k-\ell}(n+P_{k-\ell}(m))\\
&\quad \cdot \prod_{j=k-\ell+1}^k \Lambda_{W_j,\leq w_j^r}(n+P_j(m))(1-|n+P_j(m)|^{\widetilde{\beta}_j-1}\widetilde{\chi}_j(|n+P_j(m)|))\\
&\quad \quad \quad  \ll_A (\log N)^{-A},
\end{split}
\end{align}
and that for any $j\in [k]$ we have
\begin{align}\label{eq:Ejbound2}
\mathbb{E}_{n\leq N^d}\mathbb{E}_{m\leq N}|\mathcal{E}_j(n+P_j(m))|^k\ll_A(\log N)^{-A}.    
\end{align}
The bound~\eqref{eq:Ejbound2} follows by making the change of variables $n'=n+P_j(m)$ and applying Lemma~\ref{le:Eerror} with $s=r=(\log \log N)^2$. The remaining task is then to prove~\eqref{eq:goal2b}. 

For proving~\eqref{eq:goal2b}, our aim is to split $m$ and $n$ into progressions in~\eqref{eq:goal2b} to make the last product over $j$ there constant, after which Proposition~\ref{prop:GVNT3} can be applied. Expanding out the definition of $\Lambda_{W_j,\leq w_j^r}$, the estimate~\eqref{eq:goal2b} reduces to showing that
\begin{align*}
&\mathbb{E}_{n\leq N^d}\mathbb{E}_{m\leq N}\theta_1(m)\theta_2(n)\prod_{j=1}^{k-\ell-1} \Lambda(n+P_{j}(m))\cdot E_{k-\ell}(n+P_{k-\ell}(m))\\
& \quad \cdot \prod_{j=k-\ell+1}^k1_{d_j\mid n+P_j(m)}(1-|n+P_j(m)|^{\widetilde{\beta}_j-1}\widetilde{\chi}_j(|n+P_j(m)|))\\
&\quad \quad \ll_A \frac{(\log N)^{-A}}{[d_{k-\ell+1},\ldots, d_{k}]}.
\end{align*}
for any natural numbers $d_j\leq w_j^r$. 

Using~\eqref{eq:integral} and exchanging the order of integration and summation, and splitting the $n$ variable into short segments, we reduce to
\begin{align}\label{eq:goal2d}\begin{split}
&\mathbb{E}_{m\leq N}\theta_1(m)\theta_2(n)\mathbb{E}_{n\in I}\prod_{j=1}^{k-\ell-1} \Lambda(n+P_{j}(m))\cdot E_{k-\ell}(n+P_{k-\ell}(m))\prod_{j=k-\ell+1}^k1_{d_j\mid n+P_j(m)}\widetilde{\chi}_j(n+P_j(m))\\
&\quad \quad \ll_A \frac{(\log N)^{-A}}{[d_{k-\ell+1},\ldots, d_{k}]}
\end{split}
\end{align}
for any interval $I\subset [N^d/(\log N)^{A}, N^d]$ of length $N^d/(\log N)^{2A}$ and all unimodular weights $\theta_1,\theta_2\colon \mathbb{Z}\to \mathbb{C}$.

Let 
\begin{align*}
Q'= \prod_{j=k-\ell+1}^k\widetilde{q}_j,\quad D=[d_{k-\ell+1},\ldots, d_{k}].    
\end{align*}
Splitting the averaging variables in~\eqref{eq:goal2d} into residue classes $\pmod{DQ'}$, it suffices to show that for any $1\leq a_1,a_2\leq DQ'$ we have 
\begin{align}\label{eq:goal2e}\begin{split}
&\mathbb{E}_{n\in I}\mathbb{E}_{m\leq N/(DQ')}1_{n+P_1(DQ'm+a_2)\equiv a_1+P_1(a_2)\pmod{DQ'}}\\
&\quad \quad \cdot\theta_1(m)\theta_2(n)\prod_{j=1}^{k-\ell-1}\Lambda(n+P_j(DQ'm+a_2))E_{k-\ell}(n+P_{k-\ell}(DQ'm+a_2))\\
&\quad \quad \ll_A \frac{(\log N)^{-A}}{DQ'}.
\end{split}
\end{align}
By splitting the $n$ average into shorter segments if necessary, it suffices to show that for all $1\leq a_1,a_2\leq DQ'$ and $x\in [N^d]$ we have
\begin{align}\label{eq:goal2eb}\begin{split}
&\mathbb{E}_{n\leq (N/(DQ'))^d}\mathbb{E}_{m\leq N/(DQ')}1_{x+n+P_1(DQ'm+a_2)\equiv a_1+P_1(a_2)\pmod{DQ'}}\\
&\quad \quad \cdot\theta_1(m)\theta_2(n)\prod_{j=1}^{k-\ell-1}\Lambda(x+n+P_j(DQ'm+a_2))E_{k-\ell}(x+n+P_{k-\ell}(DQ'm+a_2))\\
&\quad \quad \ll_A \frac{(\log N)^{-A}}{DQ'}.
\end{split}
\end{align}

We now apply Proposition~\ref{prop:GVNT3} to the collection of $k-\ell$ polynomials $(P_1(DQ'y+a_2),\ldots,P_{k-\ell}(DQ'y+a_2))$. We use the scale $N$ in this application, absorbing $1_{m\leq N/(DQ')}$ into the $m$ weight and $1_{n\leq (N/(DQ'))^d}1_{x+n\equiv a_1\pmod{DQ'}}$ into $f_0$. We restrict the remaining input functions to $[-(C-1)N^d,(C-1)N^d]$, with $C$ sufficiently large in terms of $P_1,\ldots,P_k$, with the change of normalization introducing a factor $O((DQ')^{d+1})$. Let $m_1,\ldots, m_{k-\ell}$ be the leading coefficients of these polynomials. Then  $|m_i|\asymp_{P_1,\ldots, P_k} (DQ')^{d}$ for all $1\leq i\leq k-\ell$. In particular, this means that $|m_i|\asymp_{P_1,\ldots, P_k}|m_{k-\ell}|$ for all $1\leq i\leq k-\ell$. In addition, by the assumption on $P_i$, we know that the $m_i$ are distinct. 

Now, by Proposition~\ref{prop:GVNT3} (and the bounds $\Lambda(n)\ll \log |n|$ and~\eqref{eq:Ebound}) we see that~\eqref{eq:goal2eb} holds unless
\begin{align}\label{eq:Ek-ell}
\|E_{k-\ell}(x+\cdot)\|_{U^s[-(C-1)N^d,(C-1)N^d]}\gg  (\log N)^{-O_{A,d,k}(1)}(DQ')^{-O_{d,k}(1)}  
\end{align}
for some $1\leq s\ll_{d,k}1$ and $1\leq C\ll_{P_1,\ldots, P_k}1$. By a change of variables, the previous inequality implies that
\begin{align*}
\|E_{k-\ell}\cdot 1_{[x-(C-1)N^d,x+(C-1)N^d]}\|_{U^s[-CN^d,CN^d]}\gg  (\log N)^{-O_{A,d,k}(1)}(DQ')^{-O_{d,k}(1)-\frac{s+1}{2^s}}.  
\end{align*}
Let $B=B(d,k)$ be large. Using Vinogradov's Fourier expansion (Lemma~\ref{le:vinogradov}) to approximate $1_{[x-(C-1)N^d,x+(C-1)N^d]}$ by a trigonometric polynomial with $O((\log N)^{O_{A,d,k}(1)}(DQ')^{B})$ terms and with coefficients bounded by $O(1)$, we conclude that 
\begin{align*}
\|E_{k-\ell}\|_{U^s[-CN^d,CN^d]}\gg  (\log N)^{-O_{A,d,k}(1)}(DQ')^{-O_{d,k}(1)}.  
\end{align*}
Since $E_{k-\ell}$ is an even function, this yields 
\begin{align*}
\|E_{k-\ell}\|_{U^s[CN^d]}\gg  (\log N)^{-O_{A,d,k}(1)}(DQ')^{-O_{d,k}(1)}.  
\end{align*}
Let $c_k$ be the constant in Proposition~\ref{prop:leng2}. We may assume that $C=C(d,k)$ in~\eqref{eq:wjdef} is chosen large enough in terms of $c_k$. Then we have  
 $DQ'\ll w_{k-\ell+1}^{2r}<\exp((\log w_{k-\ell})^{c_k/2})$. 
Now, recalling that $E_{k-\ell}=\Lambda-\widetilde{\Lambda}_{W_{k-\ell}}$we obtain a contradiction with Proposition~\ref{prop:leng2}. Hence,~\eqref{eq:Ek-ell} cannot hold, completing the proof when $\deg P_i=d$ for all $1\leq i\leq k$.

For the first assertion, we can remove this restriction by a change of variables and Fourier expansion. Let $a_j$ denote the coefficient of $y^d$ in $P_j(y)$, and set
\[
a=2+\max_j|a_j|,
\qquad
\widetilde P_j(y)=P_j(y)-ay^d.
\]
Then $\deg\widetilde P_j=d$ for every $j$. Making the change of variables $n'=n+am^d$, we have
\[
n+P_j(m)=n'+\widetilde P_j(m),
\]
and the original summation region becomes
\[
1\leq m\leq N,\qquad 1\leq n'-am^d\leq N^d.
\]
We may therefore sum over the enlarged rectangle $1\leq m\leq N$, $1\leq n'\leq(a+1)N^d$, inserting the indicator $1_{[N^d]}(n'-am^d)$. Choose an integer $L\asymp N^d$ sufficiently large that reduction modulo $L$ introduces no wraparound into $[N^d]$. Fourier expansion gives
\[
1_{[N^d]}(n'-am^d)=\sum_{l\pmod L}b_le\!\left(\frac{ln'}{L}\right)e\!\left(-\frac{lam^d}{L}\right),
\qquad
\sum_{l\pmod L}|b_l|\ll\log N.
\]
 The two phases can be absorbed into $\theta_2(n')$ and $\theta_1(m)$, respectively, in~\eqref{eq:goal2}.  Consequently,  this proves the first assertion under the sole degree condition $\deg(P_i-P_j)=d$ for all $1\leq i<j\leq k$.
\end{proof}

\bibliographystyle{plain}  % bib style

\bibliography{main}

\end{document}